\documentclass[10pt, preprint]{article}
\usepackage{amssymb,amscd,amsmath, amsthm, graphicx, calc, tikz,tikz-cd,url,enumerate,stmaryrd}
\usetikzlibrary{patterns,arrows,decorations.pathmorphing}
\usepackage[all]{xy}
\usepackage[color=pink,textsize=footnotesize]{todonotes}
\usepackage{hyperref}

\def\C{\mathbb C}

\def\R{\mathbb R}

\def\Y{\mathbb Y}

\def\0{\mathbf 0}

\newtheorem{theorem}{Theorem}[section]
\newtheorem{proposition}[theorem]{Proposition}
\newtheorem{prop}[theorem]{Proposition}

\newtheorem{corollary}[theorem]{Corollary}
\newtheorem{lemma}[theorem]{Lemma}
\theoremstyle{definition}
\newtheorem{definition}[theorem]{Definition}
\newtheorem{remark}[theorem]{Remark}
\newtheorem{example}[theorem]{Example}

\DeclareMathOperator{\real}{Real}

\newcommand{\bighplus}{
  \mathop{
    \vphantom{\bigoplus} 
    \mathchoice
      {\vcenter{\hbox{\resizebox{\widthof{$\displaystyle\bigoplus$}}{!}{$\boxplus$}}}}
      {\vcenter{\hbox{\resizebox{\widthof{$\bigoplus$}}{!}{$\boxplus$}}}}
      {\vcenter{\hbox{\resizebox{\widthof{$\scriptstyle\oplus$}}{!}{$\boxplus$}}}}
      {\vcenter{\hbox{\resizebox{\widthof{$\scriptscriptstyle\oplus$}}{!}{$\boxplus$}}}}
  }\displaylimits 
}

\newcommand{\Addresses}{{
  \bigskip
  \footnotesize

  Laura Anderson, \textsc{Department of Mathematical Sciences, Binghamton University,
    Binghamton, NY 13902, USA}\par\nopagebreak
\texttt{laura@math.binghamton.edu}
    \bigskip

    James F. Davis, \textsc{Department of Mathematics, Indiana University,
    Bloomington, IN 47405, USA}\par\nopagebreak
  \texttt{jfdavis@indiana.edu}
  }}

\begin{document}

\title{Hyperfield Grassmannians}
\author{Laura Anderson
\and
James F. Davis\thanks{Partially supported by NSF grant DMS 1615056}
}
\date{}
\maketitle

\begin{abstract} In a recent paper Baker and Bowler introduced matroids over hyperfields,  offering a common generalization of matroids, oriented matroids, and linear subspaces of based vector spaces.  This  paper introduces the notion of a topological hyperfield and explores the generalization of Grassmannians and realization spaces to this context, 
particularly in relating the (hyper)fields $\R$ and $\C$ to hyperfields arising in matroid theory and in tropical geometry.
\end{abstract}

\section{Introduction}\vspace*{-4pt}

In a recent paper~\cite{Baker-Bowler} Baker and Bowler introduced
\emph{matroids over hyperfields}, a compelling unifying theory that
spans, among other things,\vspace*{-4pt}
\begin{enumerate}%
\item
matroid theory,\vspace*{-1pt}
\item
subspaces of based vector spaces, and\vspace*{-1pt}
\item
``tropical'' analogs to subspaces of vector spaces.
\end{enumerate}
A hyperfield is similar to a field, except that addition is multivalued.
Such structures may seem exotic, but, for instance, Viro has argued
persuasively \cite{Viro}, \cite{Viro2} for their naturalness
in tropical geometry, and Baker and Bowler's paper elegantly
demonstrates how matroids and related objects (e.g.~oriented matroids,
valuated matroids) can be viewed as ``subspaces of $F^{n}$, where
$F$ is a hyperfield''. Baker and Bowler's work on hyperfields was purely
algebraic and combinatorial; no topology was introduced.

The purpose of this paper is to explore topological aspects of matroids
over hyperfields, specifically Grassmannians over hyperfields and
realization spaces of matroids over hyperfields. The idea of introducing
a topology on a hyperfield is problematic from the start: as Viro
discusses in~\cite{Viro}, there are complications even in defining
the notion of a continuous multivalued function, for example, hyperfield
addition. That being said, we define a notion of a topological
hyperfield that suffices to induce topologies on the related
Grassmannians and sets of realizations, so that continuous homomorphisms
between topological hyperfields induce continuous maps between related
spaces such as Grassmannians.

We focus on Diagram~(\ref{3-by-3}) of hyperfield morphisms (which are
all continuous with respect to appropriate topologies), categorical
considerations, and the induced realization spaces and maps of
Grassmannians. Most of this collection of hyperfields was discussed at
length in~\cite{Viro} and~\cite{Viro2}. Loosely put, most
of the top half of the diagram describes the fields $\mathbb{R}$ and
$\mathbb{C}$ and their relationship to oriented matroids, phased
matroids, and matroids~-- relationships which have been extensively
studied and are known to be fraught. The bottom half of the diagram is
in some sense a tropical version of the top half. In particular, the
hyperfield $\mathcal{T }\triangle $ in the bottom row is isomorphic to
the ``tropical hyperfield'', which is called $\mathbb{Y}$
in~\cite{Viro} and captures the $(\max , +)$ arithmetic of the
tropical semifield \cite{BS}. The hyperfield $\mathcal{T }
\mathbb{C}$ in the bottom row has not been widely studied, but, as Viro
\cite{Viro} notes, algebraic geometry over this hyperfield
``occupies an intermediate position between the complex algebraic
geometry and tropical geometry''. Each of the hyperfields in the bottom
row arises from the corresponding hyperfield in the first row by
``dequantization'', a process introduced and discussed in detail
in~\cite{Viro} and reviewed briefly in
Section~\ref{sec:dequantization}. The maps going upwards in the bottom
half of the diagram are quotient maps in exactly the same way as the
corresponding maps going downwards in the top half.

With Diagrams~\eqref{3-by-3} and~\eqref{3-by-3top} we lay out a
framework for relating several spaces (Grassmannians and realization
spaces) and continuous maps between them, and with
Theorems~\ref{thm:grassmann} and~\ref{thm:realization} we show some of
these spaces to be contractible and some of these maps to be homotopy
equivalences. We find some striking differences between the top half and
bottom half of the diagram. For instance, it is well known that
realization spaces of oriented matroids over $\mathbb{R}$ can have very
complicated topology, and can even be empty. In our framework, such
spaces arise from the hyperfield morphism $\mathbb{R}\to \mathbb{S}$.
In contrast, realization spaces of oriented matroids over the
topological ``tropical real'' hyperfield $\mathcal{T }\mathbb{R}$, which
arise from the hyperfield morphism $\mathcal{T }\mathbb{R}\to
\mathbb{S}$, are all contractible.

Our enthusiasm for topological hyperfields arose from the prospect of
recasting our previous work \cite{Anderson-Davis} on combinatorial
Grassmannians in terms of Grassmannians over hyperfields. Motivated by
a program of MacPherson \cite{MacPherson}, we defined a notion of
a matroid bundle based on oriented matroids, described the process of
proceeding from a vector bundle to a matroid bundle, defined a map of
classifying spaces ${\widetilde{\mu }}: \operatorname{Gr}(r,
\mathbb{R}^{n}) \to \|\mathrm{MacP}(r,n)\|$, and showed that, stably,
${\widetilde{\mu }}$ induces a split surjection in mod 2 cohomology.
Here $\mathrm{MacP}(r,n)$ is the finite poset of rank $r$ oriented
matroids on $n$ elements and $\|\mathrm{MacP}(r,n)\|$ is its geometric
realization. One of the topological hyperfields we consider in the
current paper is the sign hyperfield $\mathbb{S}$. The poset
$\mathrm{MacP}(r,n)$ coincides as a set with the Grassmannian
$\operatorname{Gr}(r, \mathbb{S}^{n})$. We will relate the partial order
on $\mathrm{MacP}(r,n)$ to the topology on $\operatorname{Gr}(r,
\mathbb{S}^{n})$ to see that $\|\mathrm{MacP}(r,n)\|$ and $
\operatorname{Gr}(r, \mathbb{S}^{n})$ have the same weak homotopy type,
and hence the same cohomology. Further, the Grassmannian $
\operatorname{Gr}(r, \mathcal{T }\mathbb{R}_{0}^{n})$ over the tropical
real hyperfield is homotopy equivalent to $\operatorname{Gr}(r,
\mathbb{S}^{n})$. Summarizing,
%
\begin{theorem}
\label{cohomology}
\begin{enumerate}[\textit{2.}]%
\item[\textit{1.}]
There is a weak homotopy equivalence $\|\mathrm{MacP}(r,n)\| \to
\operatorname{Gr}(r,\mathbb{S}^{n})$.
\item[\textit{2.}]
There are maps of topological hyperfields
\begin{equation*}
\mathbb{R}\to \mathbb{S}\to \mathcal{T }\mathbb{R}_{0}
\end{equation*}
inducing continuous maps $ \operatorname{Gr}(r, \mathbb{R}^{n})
\to \operatorname{Gr}(r, \mathbb{S}^{n}) \to \operatorname{Gr}(r,
\mathcal{T }\mathbb{R}_{0}^{n})$. The first Grassmannian map gives a
surjection in mod 2 cohomology and the second is a homotopy equivalence.
\end{enumerate}
\end{theorem}

The Grassmannians $\operatorname{Gr}(r,\mathbb{R}^{n})$ and
$\operatorname{Gr}(r, \mathbb{C}^{n})$ are already well understood and,
indeed, of central importance in topology. The Grassmannian
$\operatorname{Gr}(r,\mathbb{S}^{n})= \mathrm{MacP}(r,n) \simeq
\operatorname{Gr}(r,\mathcal{T }\mathbb{R}_{0}^{n})$ is a space which
has been studied but remains somewhat mysterious: its topology is
discussed in Section~\ref{sec:macphersonian}. The Grassmannian
$\operatorname{Gr}(r,\mathbb{K}^{n})$ is contractible. Beyond what is
presented in this paper, the remaining spaces appear to be considerably
more difficult to understand. We hope that the discussion here
stimulates interest in pursuing the topology of these spaces. To repeat
a remark we made previously on the MacPhersonian
\cite{Anderson-Davis}: ``there are open questions everywhere you spit''.

\section{Hyperfields}

Much of the following is background material taken from
\cite{Viro} and \cite{Baker-Bowler}.

\subsection{Examples}

This section owes much to the paper of Oleg Viro \cite{Viro}.

A \emph{hyperoperation} on a set $S$ is a function from $S \times S$ to
the set of nonempty subsets of $S$. A \emph{abelian hypergroup}
$(S, \boxplus , 0)$ is a set $S$, a hyperoperation $\boxplus $ on
$S$, and an element $0 \in S$ satisfying
\begin{itemize}%
\item
For all $x,y \in S$, $x \boxplus y = y \boxplus x$.
\item
For all $x,y,z \in S$, $(x \boxplus y) \boxplus z = x \boxplus (y
\boxplus z)$.
\item
For all $x \in S$, $x \boxplus 0 = x$.
\item
For all $x \in S$, there is a unique $-x\in S$ such that $0\in x
\boxplus -x$.
\item
For all $x,y,z \in S$, $x \in y \boxplus z \Leftrightarrow -x \in -y
\boxplus -z$.
\end{itemize}

Here we define hyperoperations applied to sets in the obvious way. For
instance, $(x \boxplus y) \boxplus z$ is the union of $u \boxplus z$
over all $u \in x \boxplus y$.

The last axiom for an abelian hypergroup $S$ can be replaced by:
\begin{itemize}%
\item
\emph{Reversibility:} For all $x,y,z \in S$, $x \in y \boxplus z
\Leftrightarrow z \in x \boxplus -y$.
\end{itemize}

In the literature, an abelian hypergroup as above is sometimes called
a canonical hypergroup.

A \emph{hyperfield} is a tuple $(F, \odot , \boxplus , 1, 0)$ consisting
of a set, an operation, a hyperoperation, and two special elements
$1 \neq 0$ such that
\begin{itemize}%
\item
$(F, \boxplus , 0)$ is an abelian hypergroup.
\item
$(F-\{0\}, \odot , 1)$ is a abelian group, denoted by $F^{\times }$.
\item
For all $x \in F$, $0 \odot x = 0 = x \odot 0$.
\item
$x \odot (y \boxplus z) = (x \odot y) \boxplus (x \odot z)$.
\end{itemize}

We will often abbreviate and say that $F$ is a hyperfield.

The following property may or may not hold for a hyperfield $F$:
\begin{itemize}%
\item
\emph{Doubly distributive property:} For all $w,x,y,z \in F$,
$(w \boxplus x) \odot (y \boxplus z) = (w \odot y) \boxplus (w \odot
z) \boxplus (x \odot y) \boxplus (x \odot z)$
\end{itemize}

Suppose $K$ is a field. Here are two constructions of associated
hyperfields:
\begin{itemize}%
\item
Suppose $S$ be a subgroup of the multiplicative group of units
$K^{\times }$. Then $K/_{m} S := \{0\} \cup K^{\times }/S$ is a
hyperfield with $[a] \odot [b] = [ab]$ and $[a] \boxplus [b] = \{[c] :
c \in [a] + [b]\}$. Here $[0] = 0 \in K/_{m} S$ and for $a \in K^{
\times }$, $[a] = aS$. Note that $[a] \boxplus [b]$ is independent of
the choice of representatives for $[a]$ and $[b]$.

More generally, let $(F,\odot , \boxplus ,1,0)$ be a hyperfield and
$S$ be a subgroup of the multiplicative group of units $F^{\times }= F
- \{0\}$. Then $F/_{m} S = \{0\} \cup F^{\times }/S$ is a hyperfield
with $[a] \odot [b] = [a\odot b]$ and $[a] \boxplus [b] = \{[c] : c
\in [a] \boxplus [b]\}$.
\item
Suppose $K = K_{-} \cup \{0\} \cup K_{+}$ is an ordered field. Define
a hyperfield $F$ with $F^{\times }= K^{\times }$, but with
\begin{equation*}
a \boxplus b =
\begin{cases}
a
& \text{if } |a| > |b|
\\
a
&\text{if }a = b
\\
[-|a|,|a|]
&\text{if }a = -b
\end{cases}
\end{equation*}
\end{itemize}

A \emph{homomorphism of hyperfields} is a function $h : F \to F'$ such
that $h(0) = 0$, $h(1) = 1$, $h(x \odot y) = h(x) \odot h(y)$ and
$h(x \boxplus y) \subseteq h(x) \boxplus h(y)$.

Many examples of hyperfields and homomorphisms are encoded in the
diagram below.
%
\begin{equation}
\label{3-by-3}
\large
\begin{tikzcd}
[row sep=small] \mathbb R\arrow[r, hook] \arrow[dd,
"\operatorname{ph}"'] &\mathbb C\arrow[r, "|~|"] \arrow[d,
"\operatorname{ph}"'] & \triangle \arrow[dd, "\operatorname{ph}"'] \\
&\mathbb P\arrow[dd]\arrow[rd,"|~|"]&\\ \mathbb S\arrow[ru,
hook]\arrow[rd,hook] \arrow[dd, dashed] & & \mathbb K\arrow[dd, dashed]
\arrow[uu, shift right=1ex, dashed] \\ & \Phi \arrow[ru, "|~|"]
\arrow[d, dashed]&\\ {\mathcal T }\mathbb R\arrow[r, hook] \arrow[uu, shift
right= 1ex, "\operatorname{ph}"'] & {\mathcal T }\mathbb C\arrow[r, "|~|"]
\arrow[u, shift right=1ex, "\operatorname{ph}"']&{\mathcal T}
\triangle \arrow[uu, shift right=1ex, "\operatorname{ph}"']
\end{tikzcd}
\end{equation}

The diagram with the solid arrows commutes. The four dashed arrows are
inclusions giving sections (one-sided inverses). Here $
\operatorname{ph}$ is the \emph{phase map} $\operatorname{ph}(x) = x/|x|$
if $x \neq 0$ and $\operatorname{ph}(0) = 0$.

In each of the ten hyperfields, the underlying set is a subset of the
complex numbers closed under multiplication. And, in
each hyperfield, multiplication, the additive identity, and the
multiplicative identity coincides with that of the complex numbers.

Here are the hyperfields in the diagram:
\begin{enumerate}%
\item
$\mathbb{R}$ is the field of real numbers. $\mathbb{C}$ is the field of
complex numbers.
\item
$\triangle = (\mathbb{R}_{\geq 0}, \times , \boxplus , 1,0) $ is the
\emph{triangle hyperfield} of Viro \cite{Viro}. Here $a \boxplus
b =\{ c : |a-b| \leq c \leq a+b\}$ which can be interpreted as the set
of all numbers $c$ such that there is a triangle with sides of length
$a, b, c$. Note that the additive inverse of $a$ is $a$.
\item
$\mathbb{P}= (S^{1} \cup \{0\}, \times , \boxplus , 1, 0)$ is the
\emph{phase hyperfield}. If $a \in S^{1}$, then $a \boxplus -a = \{-a,0,a
\}$ and $a \boxplus a = a$. If $a, b \in S^{1}$ and $a \neq \pm b$,
then $a \boxplus b$ is the shortest open arc connecting $a$ and $b$.
Note that the additive inverse of $a$ is $-a$.
\item
$\mathbb{S}=(\{-1,0,1\},\times , \boxplus , 1, 0)$ is the \emph{sign
hyperfield}. Here $1 \boxplus 1 = 1$, $-1 \boxplus -1 = -1$, and
$1 \boxplus -1 = \{-1,0,1\}$. Note that the additive inverse of $a$ is
$-a$.
\item
$\Phi = (S^{1} \cup \{0\}, \times , \boxplus , 1, 0)$ is the
\emph{tropical phase hyperfield}. If $a \in S^{1}$, then $a \boxplus -a
= S^{1} \cup \{0\}$ and $a \boxplus a = a$. If $a, b \in S^{1}$ and
$a \neq \pm b$, then $a \boxplus b$ is the shortest closed arc
connecting $a$ and $b$. Note that the additive inverse of $a$ is $-a$.
\item
$\mathbb{K}= (\{0,1\},\times , \boxplus , 1, 0)$ is the \emph{Krasner
hyperfield}. Here $1 \boxplus 1 = \{0,1\}$. Note that the additive
inverse of $a$ is $a$.
\item
${\mathcal T} \mathbb{R}= (\mathbb{R},\times , \boxplus , 1, 0)$ is the
\emph{tropical real hyperfield}. Here if $|a| > |b|$, then $a \boxplus
b = a$. Also $a \boxplus a = a$ and $a \boxplus -a = [-|a|,|a|]$. Note
that the additive inverse of $a$ is $-a$. This hyperfield was studied
by Connes and Consani \cite{Connes-Consani}, motivated by
considerations in algebraic arithmetic geometry.
\item
${\mathcal T} \mathbb{C}= (\mathbb{C},\times , \boxplus , 1, 0)$ is the
\emph{tropical complex hyperfield}. One defines $a \boxplus -a = \{x
\in \mathbb{C}: |x| \leq |a| \}$, the disk of radius $|a|$ about the
origin. If $|a| > |b|$, then $a \boxplus b = a$. If $|a| = |b|$ and
$a \neq -b$, then $a \boxplus b$ is the shortest closed arc connecting
$a$ and $b$ on the circle of radius $|a|$ with center the origin. Note
that the additive inverse of $a$ is $-a$.
\item
${\mathcal T}\triangle = (\mathbb{R}_{\geq 0},\times , \boxplus , 1, 0)$ is
the \emph{tropical triangle hyperfield}. Here if $a > b$, then
$a \boxplus b = a$, and $a \boxplus a = [0, a]$. Note that the additive
inverse of $a$ is $a$.

The logarithm map from $\mathcal{T }\triangle $ to $\mathbb{R}\cup \{-
\infty \}$ induces a hyperfield structure on $\mathbb{R}\cup \{-
\infty \}$. Following \cite{Viro} we denote this hyperfield
$\mathbb{Y}$ and call it the \emph{tropical hyperfield}. In tropical
geometry it is standard to work with the \emph{tropical semifield}, whose
only difference from $\mathbb{Y}$ is that $a\boxplus b$ is defined to
be $\max (a,b)$ for all $a$ and $b$. (Thus in both the tropical
hyperfield and the tropical semifield, $-\infty $ is the additive
identity, but in the tropical semifield no real number has an additive
inverse.) Hyperfield language offers considerable advantages over
semifield language. For instance, consider a polynomial $p(x)=\sum
_{i=1}^{n} a_{i} x^{i}$ with coefficients in $\mathbb{R}$. In the
language of the tropical semifield, a root of $p$ is a value $c$ such
that the maximum value of $a_{i} c^{i}$ (under semifield multiplication)
is achieved at two or more values of $i$. The equivalent but more
natural definition in terms of $\mathbb{Y}$ is that $c$ is a root of
$p$ if $-\infty $ (the additive identity) is in $\bighplus _{i=1}^{n}
a_{i} c^{i}$ (under hyperfield multiplication and addition).
\end{enumerate}

\begin{prop}
\begin{enumerate}[\textit{2.}]%
\item[\textit{1.}]
For each hyperfield $F$ in the second column of Diagram~\eqref{3-by-3},
$S^{1}$ is a multiplicative subgroup of $F^{\times }$, and the image of
the map $|~|$ with domain $F$ is $F/_{m} S^{1}$.
\item[\textit{2.}]
For each hyperfield $F$ in the first or last row of
Diagram~\eqref{3-by-3}, $\mathbb{R}_{>0}$ is a multiplicative subgroup
of $F^{\times }$, and the image of the map $\operatorname{ph}$ with
domain $F$ is $F/_{m} \mathbb{R}_{>0}$.
\end{enumerate}
\end{prop}

All the maps in Diagram \eqref{3-by-3} are homomorphisms of hyperfields.
Note that neither identity map $\mathbb{R}\to \mathcal{T }\mathbb{R}$
nor $\mathcal{T }\mathbb{R}\to \mathbb{R}$ is a hyperfield homomorphism.
There are, in fact, no hyperfield homomorphisms from $\mathcal{T }
\mathbb{R}$ to $\mathbb{R}$, and the only hyperfield morphism from
$\mathbb{R}$ to $\mathcal{T }\mathbb{R}$ is the composition of the maps
shown in Diagram \eqref{3-by-3}. Similar remarks apply to $\mathbb{C}$
and $\mathcal{T }\mathbb{C}$.

We leave the verification of the axioms for hyperfields and hyperfield
homomorphisms in Diagram \eqref{3-by-3} to the diligent reader.

Several of the hyperfields above play special roles. The Krasner
hyperfield is a final object: for any hyperfield $F$ there is a unique
hyperfield homomorphism $F \to \mathbb{K}$, where $0$ maps to $0$ and
every nonzero element maps to $1$. As will be discussed in
Section~\ref{sec:orderings_and_norms}, the hyperfields $\mathbb{S}$,
$\triangle $, and $\mathcal{T }\triangle $ are representing objects for
the sets of orderings, norms, and nonarchimedean norms on hyperfields.

Note the vertical symmetry of Diagram~\eqref{3-by-3}.
Section~\ref{sec:dequantization} will review the idea (due to Viro) that
the last row of Diagram~\eqref{3-by-3} is obtained from the first row
via \emph{dequantization}, an operation on $\mathbb{R}$, $\mathbb{C}$,
and $\triangle $ that preserves the underlying set and multiplicative
group. Each map in the lower half of the diagram is identical, as a set
map, to its mirror image in the upper half of the diagram. Thus we think
of the lower half of the diagram as the `tropical' version of the upper
half.

The row $\mathbb{S}\hookrightarrow \mathbb{P}\to \mathbb{K}$ is the row
of most traditional interest to combinatorialists, particularly
$\mathbb{S}$, which as we shall see leads to oriented matroids, and
$\mathbb{K}$, which leads to matroids. ($\mathbb{P}$ leads to phased
matroids, introduced in~\cite{Anderson-Delucchi}.)

\subsection{Dequantization}%
\label{sec:dequantization}

We now review Viro's dequantization, which is a remarkable way of
passing from an entry $F$ in the first row of Diagram~\eqref{3-by-3} to
the corresponding entry $\mathcal{T }F$ in the last row via the identity
map, perturbing the addition in $F$ to the hyperfield addition in the
tropical hyperfield $\mathcal{T }F$.

For $h \in \mathbb{R}_{>0}$, define the homeomorphism $S_{h} : F
\to F$ by
\begin{equation*}
S_{h}(x) =
\begin{cases}
|x|^{1/h} \frac{x}{|x|}
& x \neq 0
\\
0
& x = 0
\end{cases}
\end{equation*}
Note that $S_{h}(xy) = S_{h}(x)S_{h}(y)$. Define $x +_{h} y = S_{h}
^{-1}(S_{h}(x) + S_{h}(y))$. (In the case $F = \mathbb R$, perhaps a better way to think about
$+_{h}$ is to take $h = a/b$ where $a$ and $b$ are odd positive integers
and define $x +_{h} y = (x^{1/h} + y^{1/h})^{h}$. Then for a general
$h >0$ define $+_{h}$ as a limit.)

For $ h > 0$, let $F_{h}$ be the topological field $(F,+_{h}, \times
, 1,0)$. Then $S_{h} : F_{h} \to F$ is an isomorphism of topological
fields.

It is tempting to define $x+_{0}y$ to be $\lim _{h\to 0}x+_{h}y$. But
this operation is not associative: $(1+_{0}-1)+_{0}-1=0+_{0}-1=-1
\neq 0=1+_{0}(-1+_{0} -1) $. A better approach is suggested by
considering the graph of $z=x+_{h}y$ for small values of $h$. The graph
when $F=\mathbb{R}$ and $h=1/5$ (a variant of which also appears
in~\cite{Connes-Consani}) is shown in
Fig.~\ref{fig:dequantization}.
As $h$ approaches 0, this graph resembles $\{(x,y,z): x,y\in F, z
\in \mathcal{T }F, z \in x \boxplus y\}$. Section~9.2 and Theorem~9.A
of \cite{Viro} make this precise as follows: for $F\in \{
\mathbb{R}, \mathbb{C}\}$ define $\Gamma =\{(x,y,z,h)\in F^{3}\times
\mathbb{R}_{\geq 0}:z=x +_{h} y\}$. Then define $x\boxplus y=\{z: (x,y,z,0)
\in \overline{\Gamma }\}$. This hyperaddition is, in fact, the
hyperaddition in $\mathcal{T }\mathbb{R}$ and $\mathcal{T }\mathbb{C}$.
Viro refers to this approximation of $\mathcal{T }F$ by the classical
fields $F_{h}$ as dequantization, contrasted with the usual quantization
setup of quantum mechanics where one deforms a commutative ring to a
noncommutative ring.

%
\begin{figure}
\includegraphics{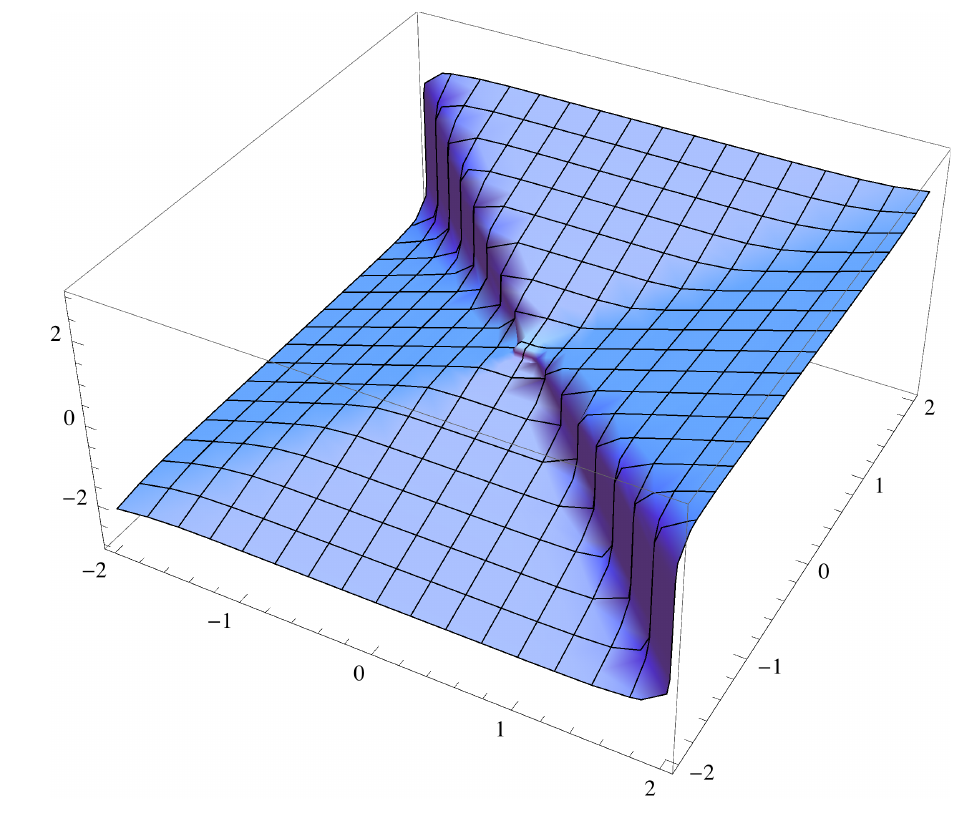}
\caption{The graph of $z = x+_{1/5} y = (x^5 + y^5)^{1/5}$.}
\label{fig:dequantization}
\end{figure}
%

\subsection{Structures on hyperfields}

\subsubsection{Orderings and norms}%
\label{sec:orderings_and_norms}

This section also owes much to the paper of Oleg Viro
\cite{Viro}.

\begin{definition}
An \emph{ordering} on a hyperfield $(F, \odot , \boxplus , 1, 0)$ is a
subset $F_{+} \subset F$ satisfying $F_{+} \boxplus F_{+} \subseteq F
_{+}$, $F_{+} \odot F_{+} \subseteq F_{+}$, and $F = F_{+} \coprod \{
0 \} \coprod -F_{+}$.
\end{definition}

An ordering on the sign hyperfield $\mathbb{S}$ is given by
$\mathbb{S}_{+} = \{1\}$. An ordering on a hyperfield $F$ determines and
is determined by a hyperfield homomorphism $h : F \to \mathbb{S}$ with
$F_{+} = h^{-1}\{1\}$.

Orderings on hyperfields pull back, so if $h : F \to F'$ is a hyperfield
homomorphism and $F'_{+}$ is an ordering on $F'$, then $h^{*}F'_{+} :=
h^{-1}F'_{+}$ is an ordering on $F$. For any hyperfield~$F$, the set of
orderings  is given by $\{ h^{*}\mathbb{S}
_{+} : h : F \to \mathbb{S}\}$, where $h$ runs over all hyperfield
homomorphisms from $F$ to $\mathbb{S}$.

Categorically speaking, the functor $O : \operatorname{Hyperfield}
^{op} \to \operatorname{Set}$ given by the set of orderings on a
hyperfield is a representable functor with representing object
$\mathbb{S}$. In other words, the pullback gives a natural bijection
from $\operatorname{Hyperfield}(-,\mathbb{S})$ to $O(-)$.

\begin{definition}
A \emph{norm} on a hyperfield $(F, \odot , \boxplus , 1, 0)$ is a
function $\|~\| : F \to \mathbb{R}_{\geq 0}$ satisfying
\begin{enumerate}%
\item
$x \neq 0$ if and only if $\|x\|>0$
\item
$\|x \odot y\| = \|x\|\|y\|$
\item
$\|x \boxplus y\| \subseteq [0, \|x\| + \|y\|]$
\end{enumerate}
\end{definition}

The identity function is a norm on the triangle hyperfield $\triangle
$.

Norms on hyperfields pull back, so if $h : F \to F'$ is a hyperfield
homomorphism and $\|~\|$ is a norm on $F'$, then $h^{*}\|~\| = \|~\|
\circ h$ is a norm on $F$. We claim that $\triangle $ is a representing
object for norms on hyperfields. The key lemma follows.\vspace*{-2pt}

\begin{lemma}
Let $F$ be a hyperfield. A function $\|~\| : F \to \mathbb{R}_{\geq 0}$
is a norm if and only if $\|~\|: F \to \triangle $ is a hyperfield
homomorphism.\vspace*{-2pt}
\end{lemma}

\begin{proof}
It is clear that if $\|~\| : F \to \triangle $ is a hyperfield
homomorphism, then $\|~\|$ is a norm.

Now suppose that $\|~\| : F \to \mathbb{R}_{\geq 0}$ is a norm. Note
that in a hyperfield, $-x \odot -x = x \odot x$, which implies that
$\|-x\| = \|x\|$. Recall that in the triangle hyperfield, $a \boxplus
b = [|a-b|, a + b]$. To prove that $\|~\| : F \to \triangle $ is a
homomorphism, the only nontrivial part is to show for $x,y,z \in F$,
that if $z \in x \boxplus y $, then $\vert \|x \| - \|y\| \vert \leq \|z\|$.
Without loss of generality, $ \|x \| \geq \|y\| $.\vspace*{-2pt}
\begin{align*}
z \in x \boxplus y
& \Longrightarrow x \in -y \boxplus z
\\[-2pt]
& \Longrightarrow \|x\| \leq \|-y\| + \|z\| = \|y\| + \|z\|
\\[-2pt]
& \Longrightarrow \|x\| - \|y\| \leq \|z\|\qedhere
\end{align*}
\end{proof}

Thus the set of norms on an arbitrary hyperfield $F$ is given by
$\{ {h^{*}\|~\|_{\triangle }}\}$, where $h$ runs over all hyperfield
homomorphisms from $F$ to $\triangle $ and $\|~\|_{\triangle } :
\triangle \to \mathbb{R}_{\geq 0}$ is the identity function.

Categorically speaking, the functor $N : \operatorname{Hyperfield}
^{op} \to \operatorname{Set}$ given by the set of norms on a hyperfield
is a representable functor with representing object $\triangle $. In
other words, the pullback gives a natural bijection from $
\operatorname{Hyperfield}(-,\triangle )$ to $N(-)$.

There is a hyperfield homomorphism $\mathbb{K}\to \triangle $. Thus any
hyperfield has a norm by sending $0$ to $0$ and any nonzero element to
$1$.\vspace*{-2pt}

\begin{definition}
A \emph{nonarchimedean norm} on a hyperfield $(F, \odot , \boxplus , 1,
0)$ is a function $\|~\| : F \to \mathbb{R}_{\geq 0}$ satisfying\vspace*{-2pt}
\begin{enumerate}%
\item
$x \neq 0$ if and only if $\|x\|>0$
\item
$\|x \odot y\| = \|x\|\|y\|$
\item
$\|x \boxplus y\| \subseteq [0, \max \{\|x\|, \|y\|\}]$\vspace*{-2pt}
\end{enumerate}
\end{definition}

The identity function is a nonarchimedean norm on the tropical triangle
hyperfield $\mathcal{T }\triangle $.

Nonarchimedean norms on hyperfields were a key motivation for the
introduction of hyperfields by Krasner \cite{Krasner}; he used the
related notion of a valued hyperfield.

Nonarchimedean norms on hyperfields pull back, so if $h : F \to F'$ is
a hyperfield homomorphism and $\|~\|$ is a nonarchimedean norm on
$F'$, then $h^{*}\|~\| = \|~\| \circ h$ is a nonarchimedean norm on
$F$. We claim that $\mathcal{T }\triangle $ is a representing object for
nonarchimedean norms on hyperfields. The key lemma follows.\vspace*{-2pt}

\begin{lemma}
Let $F$ be a hyperfield. A function $\|~\| : F \to \mathbb{R}_{\geq 0}$
is a nonarchimedean norm if and only if $\|~\|: F \to \mathcal{T }
\triangle $ is a hyperfield homomorphism.
\end{lemma}

\begin{proof}
It is clear that if $\|~\| : F \to \mathcal{T }\triangle $ is a
hyperfield homomorphism, then $\|~\|$ is a nonarchimedean norm.

Now suppose that $\|~\| : F \to \mathbb{R}_{\geq 0}$ is a nonarchimedean
norm. To prove that $\|~\| : F \to \mathcal{T }\triangle $ is a
homomorphism, one must show that if $z \in x \boxplus y$, then
$\|z\| \in \|x\| \boxplus \|y\|$. This is clear if $\|x\| = \|y\|$.

Thus let's assume $z \in x \boxplus y$ and $\|x\| > \|y\|$. We wish to
show that $\|z\| = \|x\|$.
\begin{equation*}
\|z\| \leq \max \{ \|x\|,\|y\| \} = \|x\|
\end{equation*}
By reversibility, $x \in -y \boxplus z$. Then
\begin{equation*}
\|x\| \leq \max \{ \|y\|,\|z\| \} = \|z\|\qedhere
\end{equation*}
\end{proof}

Thus the set of nonarchimedean norms on an arbitrary hyperfield $F$ is
given by $\{ {h^{*}\|~\|}_{\mathcal{T }\triangle } \}$, where $h$ runs
over all hyperfield homomorphisms from $F$ to $\mathcal{T }\triangle
$ and $\|~\|_{\mathcal{T }\triangle } : \mathcal{T }\triangle \to
\mathbb{R}_{\geq 0}$ is the identity function.

Categorically speaking, the functor $N : \operatorname{Hyperfield}
^{op} \to \operatorname{Set}$ given by the set of nonarchimedean norms
on a hyperfield is a representable functor with representing object
$\mathcal{T }\triangle $. In other words, the pullback gives a natural
bijection from $\operatorname{Hyperfield}(-,\mathcal{T }\triangle )$ to
$N(-)$.

There is a hyperfield homomorphism $\mathbb{K}\to \mathcal{T }\triangle
$. Thus any hyperfield has a nonarchimedean norm by sending $0$ to
$0$ and any nonzero element to $1$.

The function $\operatorname{ph}$ on $\mathbb{C}$ offers the motivating
example for the following definition.

\begin{definition}
An \emph{argument} on a hyperfield $(F,\odot ,\boxplus ,1,0)$ is a group
homomorphism $\arg : F^{\times }\to S^{1}$ satisfying\vspace*{-2pt}
\begin{enumerate}%
\item
$\arg (-x) = -\arg (x)$.
\item
If $\arg (x) = \arg (y)$, then $\arg (x \boxplus y) = \arg (x)$.
\item
If $\arg (x) \neq \pm \arg (y)$ and if $z \in x \boxplus y$, then
$\arg (z)$ is on the shortest open arc on the circle connecting
$\arg (x)$ and $\arg (y)$.\vspace*{-2pt}
\end{enumerate}
\end{definition}

An argument extends to a function $\arg : F \to S^{1} \cup \{0\}$ by
setting $\arg (0) = 0$.

We claim that $\mathbb{P}$ is a representing object for arguments on
hyperfields. The key lemma follows.

\begin{lemma}
Let $F$ be a hyperfield. A function $\arg : F^{\times }\to S^{1}$ is an
argument if and only if $\arg : F \to \mathbb{P}$ is a hyperfield
homomorphism.
\end{lemma}

\begin{proof}
It is clear that if $\arg : F \to \mathbb{P}$ is hyperfield
homomorphism, then $\arg : F^{\times }\to S^{1}$ is an argument.

Now suppose $\arg : F^{\times }\to S^{1}$ is an argument. We need to
show that   $\arg (x \boxplus y)
\subseteq \arg (x) \boxplus \arg (y)$, for each $x,y \in F^{\times}$. The only nontrivial case is when
$\arg (y) = \arg (-x)$. By reversibility, if an element $z \in x
\boxplus y$, then $y\in z\boxplus -x$. If $\arg (z)$ were not in
$\arg (x) \boxplus \arg (y) = \{0, \pm \arg (x)\}$, then $\arg (z
\boxplus -x)$ would be\vadjust{\goodbreak} contained in the shortest open arc connecting
$\arg (z)$ and $\arg (-x)=\arg (y)$. In particular, $\arg (z\boxplus
-x)$ would not contain $\arg (y)$, a contradiction.
\end{proof}

\begin{definition}
A \emph{$\Phi $-argument} on a hyperfield $(F,\odot ,\boxplus ,1,0)$ is
a group homomorphism $\arg : F^{\times }\to S^{1}$ satisfying\vspace*{-3pt}
\begin{enumerate}%
\item
$\arg (-x) = -\arg (x)$.
\item
If $\arg (x) = \arg (y)$, then $\arg (x \boxplus y) = \arg (x)$.
\item
If $\arg (x) \neq \pm \arg (y)$ and if $z \in x \boxplus y$, then
$\arg (z)$ is on the shortest closed arc on the circle connecting
$\arg (x)$ and $\arg (y)$.\vspace*{-3pt}
\end{enumerate}
\end{definition}

An argument is a $\Phi $-argument. Here is an example of a
$\Phi $-argument on a field which is not an argument. If $\alpha
\in \mathbb{C}$, then a $\Phi $-argument on the rational function field
$\mathbb{C}(x)$ is given by $\arg _{\alpha } ((x-\alpha )^{n}f(x)/g(x))
= \operatorname{ph}(f(\alpha )/g(\alpha ))$ where $n$ is an integer and
$f$ and $g$ are polynomials which do not have $\alpha $ as a root.

We claim that $\Phi $ is a representing object for $\Phi $-arguments on
hyperfields. The key lemma follows.\vspace*{-2pt}

\begin{lemma}
Let $F$ be a hyperfield. A function $\arg : F^{\times }\to S^{1}$ is a
$\Phi $-argument if and only if $\arg : F \to \Phi $ is a hyperfield
homomorphism.\vspace*{-2pt}
\end{lemma}

We will omit the proof of this lemma.

This above discussion show a possible utility for the notion of
hyperfield. The notions of orderings, norms, nonarchimedean norms,
arguments, and $\Phi $-arguments are important for fields, but to find
a representing object one needs to leave the category of fields.
Likewise, to find a final object one has to leave the category of fields.\vspace*{-2pt}

\begin{remark}
Jaiung Jun \cite{jjun17} has studied morphisms to $\mathbb{K}$,
$\mathbb{S}$, and $\mathbb{Y}$ from a more general algebro-geometric
viewpoint, considering not just sets $\operatorname{Hom}(A,
\mathbb{F})$ of morphisms of hyperfields for $\mathbb{F}\in \{
\mathbb{K},\mathbb{S}, \mathbb{Y}\}$, but also sets $
\operatorname{Hom}(\operatorname{Spec}(\mathbb{F}), X)$ of morphisms of
hyperring schemes, where $X$ is an algebraic variety (in the classical
sense viewed as a hyperring scheme). This recasts the set Hom(A,\,F) by
means of the functor Spec, i.e., Hom(A,\,F)=Hom(Spec F, Spec A).\vspace*{-2pt}
\end{remark}

\subsubsection{Topological hyperfields}%
\label{topological_hyperfields}

\begin{definition}
\label{top_hyper}
A \emph{topological hyperfield} is a hyperfield $(F, \odot , \boxplus
, 1, 0)$ with a topology $T$ on $F$ satisfying:\vspace*{-2pt}
\begin{enumerate}%
\item
\label{nonzero_open}
$F-\{0\}$ is open.
\item
Multiplication $F \times F \to F$ is continuous.
\item
The multiplicative inverse map $F-\{0\} \to F-\{0\}$ is continuous.\vspace*{-2pt}
\end{enumerate}
\end{definition}

Conditions (2) and (3) guarantee that $(F^{\times }, \odot , 1)$ is
a topological group.

\begin{remark}
The reader may think it is odd that we do not require addition to be
continuous. We do too! But, in our defense: we never need this, and
there are several competing notions of continuity of a multivalued map,
see Section~8 of \cite{Viro}.
\end{remark}

We wish to topologize all of the hyperfields in Diagram \eqref{3-by-3} so that all the maps are continuous and so that all the
dotted arrows are hyperfield homotopy equivalences (defined below). Of
course we give $\mathbb{R}$ and $\mathbb{C}$ their usual topologies. We
topologize $\triangle $, $\mathbb{S}$, $\mathbb{P}$ and $\mathbb{K}$ as
quotients of $\mathbb{R}$ and $\mathbb{C}$, noting that if $F$ is a
topological hyperfield and $S$ is a subgroup of $F^{\times }$, then
$F \to F/{}_{m}S$ is a morphism of topological hyperfields, where
$F/{}_{m}S$ is given the quotient topology. We topologize $\Phi $ so
that the identity map $\mathbb{P}\to \Phi $ is a homeomorphism. Thus:
\begin{itemize}%
\item
the topology on $\triangle $ is the same as its topology as a subspace
of $\mathbb{R}$.
\item
the open sets in $\mathbb{S}$ are $\emptyset $, $\{+1\}$, $\{-1\}$,
$\{+1, -1\}$, and $\{+1, -1, 0\}$.
\item
the open sets in $\mathbb{P}$ and $\Phi $ are the usual open sets in
$S^{1}$ together with the set $S^{1}\cup \{0\}$.
\item
the open sets in $\mathbb{K}$ are $\emptyset $, $\{1\}$, and
$\{0,1\}$.
\end{itemize}

Although we can give the hyperfields $\mathcal{T }\mathbb{R}$,
$\mathcal{T }\mathbb{C}$, and $\mathcal{T }\triangle $ the topologies
inherited from the complex numbers, with these topologies many of the
maps in Diagram \eqref{3-by-3} are not continuous, for example,
$\mathbb{S}\to \mathcal{T }\mathbb{R}$. Instead we will use the
$0$-coarsening described below.

Let $(F, \odot , \boxplus , 1, 0,T)$ be a topological hyperfield. Define
two new topologies ${_{0}T}$ and $T_{0}$ on $F$, called the
\emph{$0$-fine topology} and the \emph{$0$-coarse topology} respectively,
with open sets ${_{0}T} = \{0\} \cup \{ U \in T \}$ and $T_{0} = \{F
\} \cup \{ U \in T : 0 \notin U \}$. Thus in ${_{0}T}$, $\{0\}$ is an
open set, while in $T_{0}$ there are no proper open neighborhoods of
$0$. Note that both ${_{0}T}$ and $T_{0}$ depend only on the topology
on $F^{\times }$, in fact ${_{0}T}$ is the finest topology on $F$ which
restricts to a fixed topological group $F^{\times }$ and $T_{0}$ is the
coarsest. We abbreviate $(F, \odot , \boxplus , 1, 0,{_{0}T})$ by
${_{0}F}$ and $(F, \odot , \boxplus , 1, 0,T_{0})$ by $F_{0}$. For any
topological hyperfield $F$, the identity maps ${_{0}F} \to F \to F
_{0}$ are continuous. A topological hyperfield is \emph{$0$-fine} if
$F = {_{0}F}$, i.e.~$\{0\}$ is an open set. A topological hyperfield is
\emph{$0$-coarse} if $F = F_{0}$, i.e.~the only open neighborhood of
$0$ is $F$. The hyperfields $\mathbb{S}$, $\mathbb{P}$, $\Phi $, and
$\mathbb{K}$ with the topologies described above are $0$-coarse.

We then enrich Diagram \eqref{3-by-3} to a diagram of topological
hyperfields.
%
\begin{equation}
\label{3-by-3top}
\large
\begin{tikzcd}
[row sep=small] \mathbb R\arrow[r, hook] \arrow[dd,
"\operatorname{ph}"'] &\mathbb C\arrow[r, "|~|"] \arrow[d,
"\operatorname{ph}"'] & \triangle \arrow[d, "\operatorname{Id} "'] \\
&\mathbb P\arrow[dd]\arrow[rd,"|~|"]& \triangle _0 \arrow[d,
"\operatorname{ph}"'] \\ \mathbb S\arrow[ru, hook]\arrow[rd,hook]
\arrow[dd, dashed] & & \mathbb K\arrow[dd, dashed] \arrow[u, shift right
= 1ex, dashed] \\ & \Phi \arrow[ru, "|~|"] \arrow[d, dashed]&
\\ {\mathcal T
}\mathbb R_0 \arrow[r, hook] \arrow[uu, shift right= 1ex,
"\operatorname{ph}"'] & {\mathcal T }\mathbb C_0 \arrow[r, "|~|"] \arrow[u,
shift right= 1ex, "\operatorname{ph}"']&{\mathcal T} \triangle _0 \arrow[uu,
shift right= 1ex, "\operatorname{ph}"']
\end{tikzcd}
\end{equation}

All the maps in the diagram are continuous, and all the dotted arrows
are sections of their corresponding solid arrow maps. Note also that for
each solid, dotted pair, the target of the solid arrow could be
considered both as a quotient and a sub-topological hyperfield of the
domain. The authors speculate that dequantization should provide a
justification for endowing $\mathcal{T }\mathbb{R}$ with the 0-coarse
topology, but we have not been able to make this precise.

Recall that the hyperfield $\mathbb{K}$ is a final object in the
category of hyperfields, and that $\mathbb{S}$, $\triangle $,
$\mathcal{T }\triangle $, $\mathbb{P}$, and $\Phi $ are representing
objects for the set of orderings, norms, nonarchimedean norms,
arguments, and $\Phi $-arguments on a given hyperfield. Similar
considerations apply to the topological hyperfields. An \emph{ordered
topological hyperfield} $F$ is a topological hyperfield $F$ with an
ordering in which $F_{+}$ is an open set. Norms, nonarchimedean norms,
arguments, and $\Phi $-arguments on a topological hyperfield are
required to be continuous. Then $\mathbb{K}$ is a final object in the
category of topological hyperfields and $\mathbb{S}$, $\triangle $,
$\mathcal{T }\triangle $, $\mathbb{P}$, and $\Phi $ are representing
objects for the set of orderings, norms, nonarchimedean norms,
arguments, and $\Phi $-arguments on a given topological hyperfield.

\begin{definition}
A \textbf{hyperfield homotopy} between topological hyperfields $F$ and
$F'$ is a continuous function $H: F\times I\to F'$ such that, for each
$t\in I$, the function $H_{t}:F\to F'$ taking $x$ to $H(x,t)$ is a
homomorphism of hyperfields.
\end{definition}

The following result, which is hardly more than an observation, is
fundamental to all of our considerations on Grassmannians and
realization spaces associated to hyperfields in
Diagram~\eqref{3-by-3top}.
%
\begin{prop}
\label{prop:homotopies}
Let $F\in \{\triangle _{0}, \mathcal{T }\mathbb{R}_{0},\mathcal{T }
\mathbb{C}_{0},\mathcal{T }\triangle _{0}\}$. The function $H : F
\times I \to F$
\begin{equation*}
H(x,t):=
\begin{cases}
0
& x = 0
\\
x|x|^{-t}
& x \neq 0
\end{cases}
\end{equation*}
is a hyperfield homotopy.
\end{prop}

Note that in each case $H_{0}$ is the identity and $H_{1}$ maps $F$ to
$\operatorname{ph}(F)$.

\begin{proof}
Certainly in each case $H$ is continuous, and for all $x$, $y$, and
$t$, $H_{t}(x)\odot H_{t}(y)=H_{t}(x\odot y)$.

For the hyperaddition operation on $F\in \{\mathcal{T
}\mathbb{R}_{0},\mathcal{T }\mathbb{C}_{0},\mathcal{T }\triangle _{0}
\}$ it is clear that $H_{t}(x\boxplus y)= H_{t}(x)\boxplus H_{t}(y)$ for
all $x,y,t$.

For the hyperaddition operation on $\triangle _{0}$: consider $x$,
$y$, and $z$ such that $z\in x\boxplus y$. Without loss of generality
assume $x \geq y > 0$. Note that $H_{t}(x) = x^{1-t} = x^{s}$ where
$s \in [0,1]$. If $z\in x\boxplus y$, then $x-y\leq z\leq x+y$. We wish
to show $x^{s}-y^{s}\leq z^{s}\leq x^{s}+y^{s}$.

Dividing through by $x$ respectively $x^{s}$, it's enough to show that
if $1-a\leq b\leq 1+a$ and $ a ,b>0$ then $1-a^{s}\leq b^{s}\leq 1+a
^{s}$.\goodbreak

To see the first inequality, we write the hypothesis as $1\leq a+b$ and
note that:\vspace*{-2pt}
\begin{itemize}%
\item
if at least one of $a$ and $b$ is greater than or equal to 1, then at
least one of $a^{s}$ and $b^{s}$ is greater than or equal to 1, and so
$1\leq a^{s}+b^{s}$.
\item
Otherwise, since $s\in [0,1]$ we have $a^{s}\geq a$ and $b^{s}\geq b$,
so $a^{s}+b^{s}\geq a+b\geq 1$.\vspace*{-2pt}
\end{itemize}

To see the second inequality: $b\leq 1+a$ implies $b^{s}\leq (1+a)^{s}$.
Thus it suffices to show that $(1+a)^{s} \leq 1+a^{s}$, i.e., the
function $f(a)=1+a^{s}-(1+a)^{s}$ is nonnegative on $\mathbb{R}_{
\geq 0}$. Note $f(0)=0$ and $f'(a)=s((a^{s-1}-(1+a)^{s-1})$. Since
$0<a<a+1$ and $s-1\leq 0$ we have $a^{s-1}\geq (a+1)^{s-1}$, and so
$f$ is increasing on $\mathbb{R}_{\geq 0}$.
\end{proof}

\begin{remark}
Under the bijection $\log :\mathcal{T }\triangle \to \mathbb{Y}$ from
the tropical triangle hyperfield to the tropical hyperfield, the
homotopy on $\mathcal{T }\triangle _{0}$ given in
Proposition~\ref{prop:homotopies} pushes forward to a straight-line
homotopy on $\mathbb{Y}_{0}$.\vspace*{-1pt}
\end{remark}

A \emph{homotopy equivalence $h : F \to F'$ of topological hyperfields}
is a continuous hyperfield homomorphism such that there exists a
continuous hyperfield homomorphism $g : F' \to F$ and hyperfield
homotopies $G : F \times I \to F$ and $H : F' \times I \to F'$ such that
$G(x,0) = x$, $G(x,1) = g(h(x))$, $H(x,0) = x$, and $H(x,1) = h(g(x))$.
One says that $g$ and $h$ are \emph{homotopy inverses}. Topological
hyperfields $F$ and $F'$ are \emph{homotopy equivalent} if there is a
homotopy equivalence between $F$ and $F'$.

Returning to Diagram~\eqref{3-by-3}, we have the following results about
the dotted arrows.\vspace*{-2pt}

\begin{corollary}
\label{prop:4equivs}
The following inclusions are homotopy equivalences of topological
hyperfields.\vspace*{-2pt}
\begin{enumerate}[\textit{2.}]%
\item[\textit{1.}]
$\mathbb{K}\hookrightarrow \triangle _{0}$
\item[\textit{2.}]
$\mathbb{S}\hookrightarrow \mathcal{T }\mathbb{R}_{0}$
\item[\textit{3.}]
$\Phi \hookrightarrow \mathcal{T }\mathbb{C}_{0}$
\item[\textit{4.}]
$\mathbb{K}\hookrightarrow \mathcal{T }\triangle _{0}$\vspace*{-2pt}
\end{enumerate}
In each case, the homotopy inverse is $\operatorname{ph}$.\vspace*{-2pt}
\end{corollary}

\begin{proof}
In each case the composition $\operatorname{ph}\circ \text{inc}$ is the
identity. The homotopy between the identity and $\text{inc}\circ
\operatorname{ph}$ is given in Proposition~\ref{prop:homotopies}.
\end{proof}

As we shall see (Theorems~\ref{thm:grassmann} and \ref{thm:realization}), a hyperfield homotopy induces a homotopy
equivalence of Grassmannians, and for each matroid $M$, a hyperfield
homotopy between the 0-coarsenings of hyperfields induces a homotopy
equivalence of realization spaces of~$M$.

For a hyperfield $F$ and a finite set $E$, \emph{projective space
$\mathbf{P}(F^{E})$} is the quotient of $F^{E} - \{0\}$ by the scalar
action of $F^{\times }$. Of course only the cardinality of $E$ is
relevant. Cognizant of this, let $F\mathbf{P}^{n-1} = \mathbf{P}(F
^{\{1,2, \dots , n\}})$. If $F$ is a topological hyperfield then
$ F \mathbf{P}^{n-1}$ inherits a topology via the product, subspace, and
quotient topologies.
%
\begin{prop}
\label{open_map}
Let $F$ be a topological hyperfield and $U\subset F^{n}-\{(0,\ldots ,0)
\}$. If $U$ is open then its image in $F\mathbf{P}^{n-1}$ is open.
\end{prop}

\begin{proof}
Let $\pi : F^{n}-\{(0,\ldots ,0)\}\to F\mathbf{P}^{n-1}$ be the quotient
map. Since multiplication is continuous, the function $x\to \lambda
^{-1} \odot x$ from $F$ to $F$ is continuous for each $\lambda \in F
^{\times }$, and so $\lambda \odot U$ is open. Thus the union
$\bigcup _{\lambda \in F^{\times }}\lambda \odot U$ is open. But this
union is $\pi ^{-1}(\pi (U))$, and so $\pi (U)$ is open.
\end{proof}

In the next section we will define and topologize Grassmann varieties
over hyperfields.\looseness=1

\section{The Pl\"{u}cker embedding and matroids over hyperfields}

Matroids over hyperfields generalize linear subspaces of vector spaces
$F^{n}$. For a field~$F$, the Grassmannian of $r$-subspaces of
$F^{n}$ embeds into projective space via the Pl\"{u}cker embedding
$\operatorname{Gr}(r,F^{n}) \hookrightarrow F\mathbf{P}^{
\binom{n }{r}-1}$. Furthermore the image is an algebraic variety, the
zero set of homogenous polynomials called the Grassmann--Pl\"{u}cker
relations. Strong and weak $F$-matroids over a hyperfield $F$ are
defined in~\cite{Baker-Bowler} in terms of Grassmann--Pl\"{u}cker
relations. Strong and weak $F$-matroids coincide when $F$ is a field.
We review this theory in Section~\ref{sec:plucker} and then use these
ideas to define the Grassmannian of a hyperfield in
Section~\ref{sec:defn}.

\subsection{The Pl\"{u}cker embedding}%
\label{sec:plucker}

\begin{definition}
Let $W$ be an $n$-dimensional vector space over a field $F$. Let
$\operatorname{Gr}(r,W)$ be the set of $r$-dimensional subspaces of
$W$. The \emph{Pl\"{u}cker embedding} is given by
\begin{align*}
\operatorname{Gr}(r,W)
& \to \mathbf{P}(\Lambda ^{r} W)\cong F
\mathbf{P}^{\binom{n}{r}-1}
\\
V = \text{Span}(v_{1}, \dots , v_{r})
& \mapsto v_{1} \wedge \dots
\wedge v_{r}
\end{align*}
\end{definition}

The fact that $\Lambda ^{r} V$ is a rank 1 vector space shows that the
above map is independent of the choice of basis for $V$.

More naively, we can see the embedding of $\operatorname{Gr}(r, F^{n})$
into $F\mathbf{P}^{\binom{n}{r}-1}$ in matrix terms. Let $
\mathrm{Mat}(r,n)$ be the set of $r\times n$ matrices of rank $r$ over
$F$. Thus $GL_{r}(F)$ acts on $\mathrm{Mat}(r,n)$ by left
multiplication, and the quotient $GL_{r}(F)\backslash \mathrm{Mat}(r,n)$
is homeomorphic to $\operatorname{Gr}(r, F^{n})$ by the map
$GL_{r}(F)M \mapsto \mathop{\mathrm{row}}(M)$, the row space of $M$. For each
$\{i_{1},\ldots , i_{r}\}\subseteq [n] =\{1,2, \dots , n\}$ with
$i_{1}<\cdots <i_{r}$ and $M\in \mathrm{Mat}(r,n)$ let $|M_{i_{1},
\ldots , i_{r}}|$ be the determinant of the submatrix with columns
indexed by $\{i_{1},\ldots , i_{r}\}$. Consider the map $P:
\mathrm{Mat}(r,n)\to F^{\binom{n}{r}}-\{\mathbf{0}\}$ taking each
$M$ to $(|M_{i_{1}, \ldots , i_{r}}|)_{\{i_{1},\ldots , i_{r}\}\subseteq
[n]}$. Multiplication of $M$ on the left by $A\in GL_{r}(F)$ amounts to
a multiplication of $P(M)$ by the determinant of $A$, and so $P$ induces
maps
\begin{equation*}
\begin{tikzcd}
\operatorname{Gr} (r, F^n)\arrow[r,"\cong "]\arrow[rr, bend right=20, "\tilde P"]&
GL_r(F)\backslash \operatorname{Mat}(r,n)\arrow[r,"P/\sim "] &
F\mathbf P^{{n\choose r}-1}
\end{tikzcd}
\end{equation*}
The map $\tilde{P}$ is the Pl\"{u}cker embedding. It is easily seen to
be injective: each coset $GL_{r}(F)M$ has a unique element in reduced
row-echelon form, and this element can be recovered from $\tilde{P}(M)$.

A point $\vec{x}\in F^{\binom{n}{r}}$ has coordinates $x_{i_{1},
\ldots , i_{r}}$ for each $1\leq i_{1}<\cdots <i_{r}\leq n$. To give
polynomials defining the image of the Pl\"{u}cker embedding as an
algebraic variety, it is convenient to define $x_{i_{1}, \ldots , i
_{r}}$ for sequences in $[n]$ which are not necessarily increasing, by\vspace*{-1pt}
\begin{equation*}
x_{i_{1}, \ldots , i_{r}}=\mathrm{sign}(\sigma )x_{i_{\sigma (1)},
\ldots , i_{\sigma (r)}}
\end{equation*}
for each permutation $\sigma $ of $[r]$. (In particular, $x_{i_{1},
\ldots , i_{r}}=0$ if the values $i_{j}$ are not distinct.) With this
convention we have the following well-known descriptions of the
Pl\"{u}cker embedding of the Grassmannian. A proof is found in
Chapter~VII, Section~6, of \cite{Hodge-Pedoe}.

Throughout the following the notation $i_{1}, \ldots ,
\widehat{i_{k}}, \ldots , i_{r+1}$ means that the term $i_{k}$ is
omitted.\vspace*{-1pt}

\begin{prop}
\label{prop:GP}
Let $F$ be a field.\vspace*{-2pt}
\begin{enumerate}[\textit{2.}]%
\item[\textit{1.}]
\textup{\textbf{(The Grassmann--Pl\"{u}cker relations)}} The image of $\tilde{P}$
is exactly the set of $\vec{x}$ satisfying the following. For each
$I=\{i_{1}, \ldots , i_{r+1}\}\subseteq [n]$ and each $J=\{j_{1},
\ldots , j_{r-1}\}\subseteq [n]$,
%
\begin{equation}
\label{eqn:_GP}
\sum _{k=1}^{r+1} (-1)^{k} x_{i_{1}, \ldots , \widehat{i_{k}},
\ldots , i_{r+1}}x_{i_{k},j_{1}, \ldots , j_{r-1}
\vphantom{\widehat{i_{k}}}}=0
\end{equation}
\item[\textit{2.}]
\textup{\textbf{(Weak Grassmann--Pl\"{u}cker conditions)}} The image of
$\tilde{P}$ is exactly the set of $\vec{x}$ satisfying:\vspace*{-2pt}
\begin{enumerate}[\textit{(a)}]%
\item[\textit{(a)}]
$\{\{k_{1}, \ldots , k_{r}\}\subseteq [n]: x_{k_{1}, \ldots , k_{r}}
\neq 0\}$ is the set of bases of a matroid, and
\item[\textit{(b)}]
Equation~\eqref{eqn:_GP} holds for all $I$ and $J$ with $|I-J|=3$.\vspace*{-2pt}
\end{enumerate}
\end{enumerate}
\end{prop}

The second weak Grassmann--Pl\"{u}cker condition is known as the
\emph{3-term Grassmann--Pl\"{u}cker relations}.\vspace*{-1pt}
%
\begin{example}
Let $V$ be a 2-dimensional subspace of a 4-dimensional vector space
$W$. Let $\{v_{1},v_{2}\}$ be a basis for $V$ and $\{e_{1},e_{2},e
_{3},e_{4}\}$ be a basis for $W$. Let\vspace*{-1pt}
%
\begin{align}
\label{expand}
v_{1} \wedge v_{2} = \ x_{12}\ e_{1} \wedge e_{2} + x_{13}\ e_{1}
\wedge e_{3} + x_{14}\ e_{1} \wedge e_{4}
\nonumber
\\[-1pt]
+x_{34}\ e_{3} \wedge e_{4} + x_{24}\ e_{2} \wedge e_{4} + x_{23}\ e
_{2} \wedge e_{3}
\end{align}
Then since $v_{1} \wedge v_{2} \wedge v_{1} \wedge v_{2} =0$, the
coordinates satisfy the homogenous quadratic equation\vspace*{-1pt}
%
\begin{equation}
\label{GP}
x_{12}x_{34} - x_{13}x_{24} + x_{14}x_{23} = 0.
\end{equation}
Conversely, given a nonzero solution to \eqref{GP}, there are vectors
$v_{1}$ and $v_{2}$ satisfying \eqref{expand}. One can reason as
follows. One of the $x_{ij}$ is nonzero, so without loss of generality
we may assume $x_{12} = 1$. If $v_{1} = e_{1} - x_{23}e_{3} - x_{34}e
_{4}$ and $v_{2} = e_{2} + x_{13}e_{3} + x_{24}e_{4}$, then \eqref{expand} is satisfied.

Thus for a field $K$ the image of the Pl\"{u}cker embedding of
$\operatorname{Gr}(2, K^{4})$ is the projective variety given by the
3-term Grassmann--Pl\"{u}cker relations with $I=\{2,3,4\}$ and
$J=\{1\}$.\vspace*{-2pt}
\end{example}

When the field in Proposition~\ref{prop:GP} is replaced by a hyperfield,
things go haywire. There are $\vec{x}$ satisfying the hyperfield analog
of the weak Grassmann--Pl\"{u}cker conditions, but which do not satisfy
the hyperfield analog to the general Grassmann--Pl\"{u}cker relations.
The 3-term Grassmann--Pl\"{u}cker relations lead to the notion of a
\emph{weak $F$-matroid}, while the general Grassmann--Pl\"{u}cker
relations lead to the notion of a \emph{strong $F$-matroid}.\vspace*{-2pt}

\subsection{Matroids over hyperfields}%
\label{sec:defn}

\begin{definition}
\cite{Baker-Bowler} Let $F$ be a hyperfield. Let $E$ be a finite
set. A \emph{Grassmann--Pl\"{u}cker function of rank $r$ on $E$ with
coefficients in $F$} is a function $\varphi : E^{r} \to F$ such that\vspace*{-2pt}
\begin{itemize}%
\item
$\varphi $ is not identically zero.
\item
$\varphi $ is alternating.
\item
(Grassmann--Pl\"{u}cker Relations) For any $(i_{1}, \dots , i_{r+1} )
\in E^{r+1}$ and $(j_{1}, \dots , j_{r-1} ) \in E^{r-1}$,\vspace*{-7pt}
\begin{equation*}
0 \in \bighplus _{k=1}^{r+1} (-1)^{k} \varphi (i_{1}, \ldots ,
\widehat{i_{k}}, \ldots , i_{r+1})\odot \varphi ({i_{k},j_{1}, \ldots
, j_{r-1}})
\end{equation*}
\end{itemize}
\end{definition}

Two functions $\varphi _{1}, \varphi _{2} : E^{r} \to F$ are
\emph{projectively equivalent} if there exists $\alpha \in F^{\times }$
such that $\varphi _{1} = \alpha \odot \varphi _{2}$.\vspace*{-2pt}

\begin{example}
Let $E$ be a finite subset of a vector space over a field $K$ with
\mbox{$\dim _{K} \operatorname{Span}E = r$}. Then a Grassmann--Pl\"{u}cker
function of rank $r$ on $E$ with coefficients in the Krasner hyperfield
$\mathbb{K}$ is given by defining $\varphi (i_{1}, \dots , i_{r})$ to
be zero if $\{i_{1}, \dots , i_{r}\}$ is linearly dependent and one if
$\{i_{1}, \dots , i_{r}\}$ is linearly independent. The projective
equivalence class determines a rank $r$ matroid for which $E$ is a
representation over $K$. (Because $|\mathbb{K}^{\times }|=1$, each
projective equivalence class has only one element.)\vspace*{-2pt}
\end{example}

\begin{definition}
\cite{Baker-Bowler} A strong \emph{$F$-matroid of rank $r$ on
$E$} is the projective equivalence class of a Grassmann--Pl\"{u}cker
function of rank $r$ on $E$ with coefficients in $F$.\vspace*{-2pt}
\end{definition}

\begin{definition}
\cite{Baker-Bowler} Let $F$ be a hyperfield. Let $E$ be a finite
set. A \emph{weak Grassmann--Pl\"{u}cker function of rank $r$ on $E$ with
coefficients in $F$} is a function $\varphi : E^{r} \to F$ such that\vspace*{-2pt}
\begin{itemize}%
\item
$\varphi $ is not identically zero.
\item
$\varphi $ is alternating
\item
The sets $\{i_{1}, \dots , i_{r}\}$ for which $\varphi (i_{1}, \dots
i_{r}) \neq 0$ form the set of bases of a ordinary matroid.
Equivalently, if $\kappa : F \to \mathbb{K}$ is the unique hyperfield
homomorphism to the Krasner hyperfield, then $\kappa \circ \varphi $ is
a Grassmann--Pl\"{u}cker function of rank $r$ on $E$ with coefficients
in $\mathbb{K}$.
\item
(3-term Grassmann--Pl\"{u}cker Relations) For any $I= (i_{1}, \dots , i
_{r+1} ) \in E^{r+1}$ and $J = (j_{1}, \dots , j_{r} ) \in E^{r-1}$ for
which $|I-J| = 3$,\vspace*{-1pt}
\begin{equation*}
0 \in \bighplus _{k=1}^{r+1} (-1)^{k} \varphi (i_{1}, \ldots ,
\widehat{i_{k}}, \ldots , i_{r+1})\odot \varphi ({i_{k},j_{1}, \ldots
, j_{r-1}})
\end{equation*}
\end{itemize}
\end{definition}

\begin{definition}
\cite{Baker-Bowler} A weak \emph{$F$-matroid of rank $r$ on $E$}
is the projective equivalence class of a weak Grassmann--Pl\"{u}cker
function of rank $r$ on $E$ with coefficients in $F$.
\end{definition}

Note that a strong $F$-matroid is a weak $F$-matroid. Baker and Bowler
show that if a hyperfield satisfies the doubly distributive property,
then weak and strong $F$-matroids coincide. The hyperfields in
Diagram~\eqref{3-by-3} which are doubly distributive are $\mathbb{R}$
and $\mathbb{C}$ (because they are fields), $\mathbb{S}$ and
$\mathbb{K}$ (Section~4.5 in~\cite{Viro}), $\mathcal{T }
\mathbb{R}$ (Section~7.2 in~\cite{Viro}), and $\mathcal{T }
\triangle $ (Section~5.2 in~\cite{Viro}). For each of the
remaining hyperfields, strong matroids and weak matroids do not
coincide:
\begin{itemize}%
\item
Example 3.30 in~\cite{Baker-Bowler} is a weak $\triangle $-matroid
which is not strong.
\item
Example 3.31 in~\cite{Baker-Bowler}, which is due to Weissauer,
is a function $\varphi $ from 3-tuples from a 6-element set to
$S^{1}\cup \{0\}$. Viewing $S^{1}\cup \{0\}$ as contained in the
underlying set of a hyperfield $F\in \{\mathbb{P},\Phi ,\mathcal{T }
\mathbb{C}\}$, this $\varphi $ is the Grassmann--Pl\"{u}cker function of
a weak $F$-matroid which is not strong.
\end{itemize}

When dealing with hyperfields for which weak and strong matroids
coincide, we will leave out the adjectives ``weak'' and ``strong.''

A $\mathbb{K}$-matroid is a matroid and an $\mathbb{S}$-matroid is an
oriented matroid. As we have seen, when $F$ is a field, an $F$-matroid
is the image of a subspace of $F^{E}$ under the Pl\"{u}cker embedding.

It is also possible to interpret strong $F$-matroids as generalizations
of linear subspaces in a more direct way. Associated to a strong
$F$-matroid of rank $r$ on $E$ is a set ${\mathcal V}^{*}\subseteq F^{E}$,
called the set of \emph{$F$-covectors} of the $F$-matroid. If $F$ is a
field, and thus an $F$-matroid is the Pl\"{u}cker embedding of a
subspace $V$ of $F^{E}$, then the $F$-covectors of the $F$-matroid are
exactly the elements of $V$. See~\cite{Anderson:covectors} for
details.

\section{Hyperfield Grassmannians}

\begin{definition}
The \emph{strong Grassmannian} $\operatorname{Gr}^{s}(r,F^{E})$ is the
set of strong $F$-matroids of rank $r$ on $E$. The \emph{weak
Grassmannian $\operatorname{Gr}^{w}(r,F^{E})$} is the set of weak
$F$-matroids of rank $r$ on $E$. We use the notation $
\operatorname{Gr}^{*}(r,F^{n})$, where $* \in \{s,w\}$.
\end{definition}

\begin{remark}
$\operatorname{Gr}^{s}(r,F^{E})$ is called the \emph{$F$-Grassmannian}
in~\cite{Baker-Bowler}.
\end{remark}

\begin{remark}
We abbreviate $\operatorname{Gr}^{*}(r,F^{\{1, \dots , n\}})$ by
$\operatorname{Gr}^{*}(r,F^{n})$. If $E$ has cardinality $n$, then
introducing a total order on $E$ gives a bijection $\operatorname{Gr}
^{*}(r,F^{E}) \cong \operatorname{Gr}^{*}(r,F^{n})$.
\end{remark}

Note that
\begin{equation*}
\operatorname{Gr}^{s}(r, F^{n})\subset \operatorname{Gr}^{w}(r,F^{n})
\subset \mathbf{P}(F^{n^{r}})
\end{equation*}

If $F$ is a topological hyperfield, then each of the sets above inherits
a topology. For a topological hyperfield there is a stabilization
embedding
\begin{align*}
\operatorname{Gr}^{*}(r,F^{n})
& \hookrightarrow \operatorname{Gr}
^{*}(r,F^{n+1})
\\
[\varphi : \{1, \dots , n\}^{r} \to F]
& \mapsto [{\widehat{\varphi }}:
\{1, \dots , n+1\}^{r} \to F]
\end{align*}
given by considering $\{1, \dots , n\}^{\{1, \dots , r\}}\subset \{1,
\dots , n+1\}^{\{1, \dots , r\}}$ and defining ${\widehat{\varphi }}$
to be equal to $\varphi $ on the subset and zero on the complement. We
then define $\operatorname{Gr}^{*}(r,F^{\infty })$ as the colimit of
$\operatorname{Gr}^{*}(r,F^{n})$ as $n \to \infty $. In other words,
$\operatorname{Gr}^{*}(r,F^{\infty })$ is the union of $
\operatorname{Gr}^{*}(r,F^{n})$, and a subset of $\operatorname{Gr}
^{*}(r,F^{\infty })$ is open if and only if its intersection with
$\operatorname{Gr}^{*}(r,F^{n})$ is open for all~$n$.

A continuous homomorphism $h : F \to F'$ of topological hyperfields
induces a continuous map
\begin{align*}
\operatorname{Gr}^{*}(h) : \operatorname{Gr}^{*}(r,F^{n})
&\to
\operatorname{Gr}^{*}(r,F^{\prime \,n})
\\
[\varphi ]
& \mapsto [h \circ \varphi ]
\end{align*}

The following theorem follows from the definitions, but is nonetheless
powerful.

\begin{theorem}
\label{thm:grassmann}
Let $H : F \times I \to F'$ be a hyperfield homotopy. Define
\begin{equation*}
\operatorname{Gr}^{*}(H) : \operatorname{Gr}^{*}(r,F^{n}) \times I
\to \operatorname{Gr}^{*}(r,F^{\prime \,n})
\end{equation*}
by $\operatorname{Gr}^{*}(H)([\varphi ],t) = \operatorname{Gr}^{*}(H
_{t})([\varphi ])$. Then $\operatorname{Gr}^{*}(H)$ is continuous.
\end{theorem}

If $F$ and $F'$ are homotopy equivalent, then so are the topological
spaces $\operatorname{Gr}^{*}(r,F^{n})$ and
$\operatorname{Gr}^{*}(r,F^{\prime \,n})$.

From Proposition~\ref{prop:4equivs} we see:
%
\begin{corollary}
\label{cor:grassmannhomotopy}
1. $\operatorname{Gr}(r,\mathbb{S}^{n}) \simeq \operatorname{Gr}(r,
\mathcal{T }\mathbb{R}_{0}^{n})$

2. $\operatorname{Gr}^{*}(r, \Phi ^{n})\simeq \operatorname{Gr}^{*}(r,
\mathcal{T }\mathbb{C}_{0}^{n})$ for $*\in \{s,w\}$

3. $\operatorname{Gr}(r, \mathbb{K}^{n})\simeq \operatorname{Gr}^{*}(r,
\triangle _{0}^{n})$ for $*\in \{s,w\}$

4. $\operatorname{Gr}(r, \mathbb{K}^{n})\simeq \operatorname{Gr}(r,
\mathcal{T }\triangle _{0}^{n})$
\end{corollary}

Part 1 of this corollary gives us the homotopy equivalence referred to
in Part 2 of Theorem~\ref{cohomology}. As will be discussed in
Section~\ref{sec:posets}, $\operatorname{Gr}(r, \mathbb{K}^{n})$ is
contractible, and $\operatorname{Gr}(r,\mathbb{S}^{n})$ is the
\emph{MacPhersonian} $\mathrm{MacP}(r,n)$.

\section{Realization spaces}

For any morphism of hyperfields $f: F\to F'$ and $*\in \{s,w\}$, we get
a partition of $\operatorname{Gr}^{*}(r, F^{n})$ into preimages under
$\operatorname{Gr}^{*}(f)$. The preimage of $M\in \operatorname{Gr}
^{*}(r, F^{\prime \,n})$ will be denoted $\operatorname{Real}_{F}^{*}(M)$, and
if $F$ is a topological hyperfield then $\operatorname{Real}_{F}^{*}(M)$
is the (strong or weak) \emph{realization space} of $M$ over $F$. An
element of $\operatorname{Real}_{F}^{*}(M)$ is called a (strong or weak)
\emph{realization} of $M$ over $F$. An $F'$-matroid is (strong or weak)
\emph{realizable} over $F$ if it has a (strong or weak) realization over
$F$. (This is a rephrasing of Definition 4.9
in~\cite{Baker-Bowler}.) When dealing with hyperfields for which
weak and strong coincide, we will leave out the adjectives ``weak'' and
``strong'' and write simply $\operatorname{Real}_{F}(M)$.

\begin{example}
Recall that a $\mathbb{K}$-matroid is simply called a matroid. For any
hyperfield $F$, there is a unique morphism $\kappa : F\to \mathbb{K}$.
We call the resulting partition of $\operatorname{Gr}^{*}(r, F^{n})$ the
\emph{matroid partition}. The realization space $\operatorname{Real}
_{F}^{*}(M)$ is exactly the set of (strong or weak) rank $r$
$F$-matroids on $[n]$ whose Grassmann--Pl\"{u}cker functions are nonzero
exactly on the ordered bases of the matroid $M$.
\end{example}

\begin{example}
If $F=\mathbb{S}$ and $M$ is a matroid then $\operatorname{Real}_{F}(M)$
is the set of orientations of $M$; this has been much-studied
(cf.~Section~7.9 in~\cite{BLSWZ}).
\end{example}

\begin{example}
Let $M$ be a matroid and let $\mathbb{Y}$ be the tropical hyperfield,
i.e., the hyperfield on elements $\mathbb{R}\cup \{-\infty \}$ with
operations induced by the bijection $\log :\mathcal{T }\triangle
\to \mathbb{R}\cup \{-\infty \}$. Then $\operatorname{Real}_{
\mathbb{Y}}(M)$ is essentially the \emph{Dressian} $\mathrm{Dr}_{M}$
discussed in~\cite{Maclagan-Sturmfels}. More precisely, assume
$M$ is a matroid $\varphi : E^{r}\to \mathbb{K}$ and ${\mathcal B}=\{\{i
_{1}, \ldots i_{r}\}:\varphi (i_{1}, \ldots , i_{r})=1\}$. (In standard
matroid language, $\mathcal B$ is the set of bases of $M$.) Then we have an
embedding
%
\begin{align}
\operatorname{Real}_{\mathbb{Y}}(M)
&\to \mathbb{R}^{\mathcal B}/
\mathbb{R}(1,\ldots , 1)\\
[\hat{\varphi }]&\to [(
\hat{\varphi }(i_{1}, \ldots , i_{r}): \{i_{1}, \ldots , i_{r}\}
\in \mathcal B)]
\end{align}
\begin{align}
	\real_{\Y}(M)&\to \R^{\cal B}/\R(1,\ldots, 1)\\
	[\hat\varphi]&\to [(\hat\varphi(i_1, \ldots, i_r): \{i_1, \ldots, i_r\}\in\cal B)]
	\end{align}
whose image is $\mathrm{Dr}_{M}$.

In particular, if $M$ is the uniform rank $r$ matroid on $n$ elements
then $\operatorname{Real}_{\mathbb{Y}}(M)={\mathrm{D}r}(r,n)$. (Recall
that the \emph{uniform} matroid of rank $r$ on elements $E$ is the
matroid in which each $r$-element subset of $E$ is a basis, or
equivalently, the $\mathbb{K}$-matroid whose Grassmann--Pl\"{u}cker
function $\varphi : E^{r}\to \mathbb{K}$ takes every $r$-tuple of
distinct elements to 1.)
\end{example}

Realization spaces of matroids and oriented matroids over fields is a
rich subject (cf.~Ch.~6 in~\cite{Oxley},
\cite{ZieglerOMtoday}, \cite{Ruiz}). Even determining whether a
given matroid or oriented matroid is realizable over a particular field
is a nontrivial task. Most matroids and most oriented matroids are not
realizable over $\mathbb{R}$ (Corollary 7.4.3 in~\cite{BLSWZ}),
and the realization space of an oriented matroid over $\mathbb{R}$ (or
a phased matroid over $\mathbb{C}$) can have horrendous topology
\cite{Mnev}.

\subsection{Topology of realization spaces}

\begin{lemma}
\label{lhomeo}
Let $f:F \to F'$ be a morphism of hyperfields. Let $F_{1}$ and
$F_{2}$ be topological hyperfields with underlying hyperfield $F$ such
that the identity map $F_{1}\to F_{2}$ restricts to a homeomorphism on
each preimage $f^{-1}(a)$. Let $M\in \operatorname{Gr}^{*}(r, F^{\prime \,n})$.
Then the identity maps
\begin{equation*}
\operatorname{Real}^{*}_{F_{1}}(M) \to \operatorname{Real}^{*}_{F_{2}}(M)
\end{equation*}
are homeomorphisms.
\end{lemma}

\begin{proof}
Consider a basis $\{b_{1}, \ldots , b_{r}\}$ for $M$, i.e. a set such
that each Grassmann--Pl\"{u}cker function for $M$ is nonzero on
$(b_{1}, \ldots , b_{r})$. Fix $\varphi '$ to be the unique
Grassmann--Pl\"{u}cker function for $M$ such that $\varphi '(b_{1},
\ldots , b_{r})=1$. The map $\varphi \to [\varphi ]$ from $P=\{\varphi
\in F^{[n]^{r}}: [\varphi ]\in \operatorname{Real}_{F}(M)\mbox{ and }
\varphi (b_{1}, \ldots , b_{r})=1\}$ to $\operatorname{Real}_{F}(M)$ is
a homeomorphism, by Proposition~\ref{open_map}. But $P$ is a subset of
the Cartesian product
\begin{equation*}
\prod _{(c_{1}, \ldots , c_{r})\in [n]^{r}}f^{-1}(\varphi '(c_{1},
\ldots ,c_{r}))
\end{equation*}
whose topology is the same in $F_{1}$ and $F_{2}$.
\end{proof}

\begin{corollary}
\label{chomeo}
Let $f:F \to F'$ be a morphism of hyperfields, let $M$ be an
$F'$-matroid, and let $F$ be a topological hyperfield. Then the identity
maps
\begin{equation*}
\operatorname{Real}^{*}_{{_{0}F}}(M) \to \operatorname{Real}^{*}_{F}(M)
\to \operatorname{Real}^{*}_{F_{0}}(M)
\end{equation*}
are homeomorphisms.
\end{corollary}

\begin{proof}
For $a \in F'$, either $f^{-1}(a) = \{0\}$ or $f^{-1}(a) \subseteq F -
\{0\}$. If $F$ is a topological hyperfield, then the three topologies
${}_{0}F$, $F$, and $F_{0}$ induce the same topology on $F-\{0\}$ and
on $\{0\}$, and hence on $f^{-1}(a)$. Lemma~\ref{lhomeo} applies.
\end{proof}

\begin{theorem}
\label{thm:realization}
Let $F$ and $F'$ be topological hyperfields and let $F''$ be a
hyperfield. Let $f:F\to F''$ and $f':F'\to F''$ be hyperfield morphisms.
Let $M$ be a $F''$-matroid. If $H: F_{0} \times I \to F'_{0}$ is a
hyperfield homotopy such that $f'(H(x,t))=f(x)$ for all $x$ and $t$,
then $H$ induces a homotopy $ \operatorname{Real}_{F}^{*}(M)\times I
\to \operatorname{Real}_{F'}^{*}(M)$.
\end{theorem}

\begin{proof}
We have a commutative diagram
%
\begin{equation}
\begin{tikzcd}
\operatorname{Gr} ^*(r, F_0^n)\times I\arrow[rr,"\operatorname{Gr} ^*(H)"]
\arrow[rd]&&\operatorname{Gr} ^*(r, F_0')\arrow[ld, "\operatorname{Gr} ^*(f')"]\\
&\operatorname{Gr} ^*(r,F''^n)&
\end{tikzcd}
\end{equation}
where $\operatorname{Gr}^{*}(H)$ is the homotopy of
Theorem~\ref{thm:grassmann} and the southeast map sends each
$(M,t)$ to $\operatorname{Gr}^{*}(f)(M)$. Thus
$\operatorname{Gr}^{*}(H)$ restricts to a map $\operatorname{Real}
_{F_{0}}(M)\times I \to \operatorname{Real}_{F_{0}'}(M)$, which by
Corollary~\ref{chomeo} is our desired homotopy.
\end{proof}

Applying this to the homotopies in Proposition~\ref{prop:homotopies},
and in stark contrast to the situation with realizations over
$\mathbb{R}$ and $\mathbb{C}$, we have the following.

\begin{corollary}
1. For any matroid $M$, $\operatorname{Real}_{ \triangle }^{*}(M)$ and
$\operatorname{Real}_{\mathcal{T }\triangle }(M)$ are contractible.

2. For any oriented matroid $M$, $\operatorname{Real}_{\mathcal{T }
\mathbb{R}}(M)$ is contractible.

3. For every $M\in \operatorname{Gr}^{*}(r,\Phi ^{n})$, $
\operatorname{Real}_{\mathcal{T }\mathbb{C}}^{*}(M)$ is contractible.
\end{corollary}

In particular, each of these realization spaces is nonempty: that is,
every matroid is realizable over $\triangle $ and $\mathcal{T }\triangle
$, every oriented matroid is realizable over $\mathcal{T }\mathbb{R}$,
and every (strong or weak) phased matroid is realizable over
$\mathcal{T }\mathbb{C}$. This much is actually easy to see even without
Theorem~\ref{thm:realization}. For instance, notice that $\mathbb{K}$
is a subhyperfield of $\triangle $, and so the Grassmann--Pl\"{u}cker
function of a $\mathbb{K}$-matroid $M$ can also be viewed as the
Grassmann--Pl\"{u}cker function of a $\triangle $-matroid realizing
$M$, and likewise for $\mathbb{K}\subset \mathcal{T }\triangle $,
$\mathbb{S}\subset \mathcal{T }\mathbb{R}$, and $\Phi \subset \mathcal{T
}\mathbb{C}$.

\subsection{Gluing realization spaces}

One approach to understanding the topology of a Grassmannian
$Gr^{*}(r, F^{n})$ is to understand the realization spaces arising from
a hyperfield morphism $F\to F'$ and then to understand how these
realization spaces ``glue together'' -- that is, to understand
intersections $\overline{\operatorname{Real}^{*}_{F}(M)}\cap
\operatorname{Real}^{*}_{F}(M')$. Here $\overline{S}$ denotes the
topological closure of a set $S$. Two key observations are contained in
the following.

\begin{prop}
\label{prop:glue}
1. Let $f:F\to F'$ be a morphism of topological hyperfields and
$S\subseteq \operatorname{Gr}^{*}(r, F^{\prime \,n})$. Then
\begin{equation*}
\overline{\bigcup _{M\in S}\operatorname{Real}^{*}_{F}(M)}\subseteq
\bigcup _{M'\in \overline{S}}\operatorname{Real}^{*}_{F}(M').
\end{equation*}

2. There is an $M\in \operatorname{Gr}(3,\mathbb{S}^{7})$ such that
\begin{equation*}
\emptyset \neq \overline{\operatorname{Real}_{\mathbb{R}}(M)}\subsetneq
\bigcup _{M'\in \overline{\{M\}}}\operatorname{Real}_{\mathbb{R}}(M').
\end{equation*}
\end{prop}

\begin{proof}
1. $\operatorname{Gr}^{*}(r, F^{n})-\bigcup _{M'\in \overline{S}}
\operatorname{Real}^{*}_{F}(M)$ is the preimage of the open set
$\operatorname{Gr}^{*}(r, F^{\prime \,n})-\overline{S}$ under the continuous map
$\operatorname{Gr}^{*}(f)$, hence is open. Thus
$\bigcup _{M'\in \overline{S}}\operatorname{Real}^{*}_{F}(M)$ is a closed
set containing $\bigcup _{M\in S}\operatorname{Real}^{*}_{F}(M)$.

2. Figure 2.4.4 in~\cite{BLSWZ} shows oriented matroids
$M_{1},M_{2}\in G(3,\mathbb{S}^{7})$ such that $M_{2}\in \overline{\{M
_{1}\}}$ but $\emptyset \neq \operatorname{Real}_{\mathbb{R}}(M_{2})
\nsubseteq \overline{\operatorname{Real}_{\mathbb{R}}(M_{1})}$.
\end{proof}

The first part of Proposition~\ref{prop:glue} suggests an appealing
approach when $F'$ is $\mathbb{K}$ or $\mathbb{S}$, since closures in
$\operatorname{Gr}(r, \mathbb{K}^{n})$ and $\operatorname{Gr}(r,
\mathbb{S}^{n})$ have simple combinatorial descriptions: each of these
Grassmannians is finite and has a poset structure, described in
Section~\ref{sec:posets}, and the closure of $\{M\}$ is just the set of
elements less than or equal to $M$. Thus one can hope to construct
$\operatorname{Gr}^{*}(r, F^{n})$ by attaching realization spaces of
successively greater $F'$-matroids, similarly to how one constructs a
CW complex by constructing successively higher-dimensional skeleta.
However, as the second part of the proposition suggests, the attaching
maps can be messy and are poorly understood.

If $F_{1}$ is a coarsening of $F_{2}$, $F_{2}\stackrel{\mathrm{id}}{
\to }F_{1}\to F$ are morphisms of topological hyperfields, and $M$ is
an $F$-matroid, then $\overline{\operatorname{Real}^{*}_{F_{1}}(M)}
\supseteq \overline{\operatorname{Real}^{*}_{F_{2}}(M)}$. Thus
realization spaces are ``glued more'' over $F_{1}$ -- i.e., $\overline{
\operatorname{Real}_{F_{1}}(M)}\cap \operatorname{Real}_{F_{1}}(M')
\supseteq \overline{\operatorname{Real}_{F_{2}}(M)}\cap
\operatorname{Real}_{F_{2}}(M')$ for all $M$ and $M'$ -- and
$F_{1}$ offers a better chance than $F_{2}$ of satisfying $\overline{
\operatorname{Real}^{*}_{F_{i}}(M)}=\bigcup _{M'\in \overline{\{M\}}}
\operatorname{Real}_{F_{i}}(M)$. In particular, if $F$ is a topological
hyperfield then the Grassmannians $\operatorname{Gr}^{*}(r, {_{0}F}
^{n})$, $\operatorname{Gr}^{*}(r, F^{n})$, and $\operatorname{Gr}^{*}(r,
F_{0}^{n})$ are constructed from the same realization spaces, glued
together most strongly in $\operatorname{Gr}^{*}(r, F_{0}^{n})$, and,
as Proposition~\ref{prop:0open} will show, not glued together at all in
$\operatorname{Gr}^{*}(r, {_{0}F}^{n})$. This suggests the possibility
of studying the partition of a Grassmannian into realization spaces by
tinkering with neighborhoods of 0 to improve the gluing.

The realization space of the uniform matroid of rank $r$ on elements
$E$ is just the set of all elements of $\operatorname{Gr}^{*}(r, F
^{E})$ with all Pl\"{u}cker coordinates nonzero. This implies that if
$M$ is the uniform matroid, then the realization space $
\operatorname{Real}_{F}^{*}(M)$ is open in the $F$-Grassmannian. (This
set may be empty: for instance, the uniform rank 2 matroid on elements
$\{1,2,3,4\}$ is not realizable over the field $\mathbb{F}_{2}$.) As the
following proposition shows, either this is the only open part of the
matroid partition or every part is open (and hence $\operatorname{Gr}
^{*}(r, F^{E})$ is a topological disjoint union of the parts of
$\operatorname{Real}_{F}^{*}(M)$).\vspace*{-2pt}

\begin{prop}
\label{prop:0open}
Let $F$ be a topological hyperfield. The following are equivalent.\vspace*{-2pt}
\begin{enumerate}[\textit{(a)}]%
\item[\textit{(a)}]
$\{0\}$ is open in $F$.
\item[\textit{(b)}]
For every rank $r$ matroid $M$ on $n$ elements, the realization space
$\operatorname{Real}_{F}^{*}(M)$ is open in the $F$-Grassmannian
$\operatorname{Gr}^{*}(r,F^{n})$.
\item[\textit{(c)}]
For every rank $r$ matroid $M$ on $n$ elements, $\overline{
\operatorname{Real}_{F}^{*}(M)}=\operatorname{Real}_{F}^{*}(M)$.
\item[\textit{(d)}]
There exists an $r$ and $n$ and a nonuniform rank $r$ matroid $M$ on
$n$ elements, such that $\operatorname{Real}_{F}^{*}(M)$ is nonempty and
open in the $F$-Grassmannian $\operatorname{Gr}^{*}(r,F^{n})$.\vspace*{-2pt}
\end{enumerate}
\end{prop}

Examples of topological hyperfields in which $\{0\}$ is open are
$\mathbb{P}$ and $\Phi $ topologized as subspaces of $\mathbb{C}$.

\begin{lemma}
\label{lem:rplus1}
Let $F$ be a hyperfield and $A=\{a_{1}, \ldots , a_{r+1}\}$.
\begin{enumerate}[\textit{2.}]%
\item[\textit{1.}]
Every nonzero alternating function $A^{r} \to F$ is a
Grassmann--Pl\"{u}cker function.
\item[\textit{2.}]
The function $\operatorname{Gr}(r, F^{A})\to F\mathbf{P}^{r}$ taking
$[\varphi ]$ to $[\varphi (a_{1}, \ldots , \hat{a_{i}}, \ldots , a
_{r+1})]_{ i\in [r+1]}$ is a homeomorphism.
\end{enumerate}
\end{lemma}

\begin{proof}
1. In this case the Grassmann--Pl\"{u}cker relations are trivial.

2. Let $\operatorname{Alt}(A^{r},F)$ be the set of alternating functions
from $A^{r}$ to $F$, topologized as a subset of $F^{A^{r}}$. There are
inverse continuous maps
\begin{align*}
\operatorname{Alt}(A^{r},F)
& \cong F^{A}
\\
\varphi
& \mapsto (a_{i} \mapsto \varphi (a_{1}, \ldots ,
{\widehat{a_{i}}}, \ldots , a_{r+1}))
\\
(\varphi (a_{1}, \ldots , {\widehat{a_{i}}}, \ldots , a_{r+1}) = f(a
_{i}) )
& \mapsfrom f,
\end{align*}
which descend to homeomorphisms after projectivizing.
\end{proof}

\begin{lemma}
\label{lemma:_prescribe}
Let $F$ be a topological hyperfield, $M$ a rank $r$ matroid on a finite
set $E$ such that $\operatorname{Real}_{F}^{*}(M)$ is nonempty, and
$A\subseteq E$ an $(r+1)$-element subset. Then any alternating function
$\tilde{\varphi }:A^{r}\to F$ which is nonzero exactly on the ordered
bases of $M$ contained in $A$ extends to a Grassmann--Pl\"{u}cker
function $\varphi : E^{r}\to F$ such that $[\varphi ]\in
\operatorname{Real}_{F}^{*}(M)$.
\end{lemma}

\begin{proof}
Let $A = \{a_{1}, \dots , a_{r+1}\}$. Let $\varphi _{0}: E^{r}\to F$ be
the Grassmann--Pl\"{u}cker function of some element of $
\operatorname{Real}_{F}^{*}(M)$. Let
\begin{equation*}
D=\bigodot _{i}\frac{ \varphi _{0}(a_{1},\ldots , \widehat{a_{i}},
\ldots , a_{r+1})}{\tilde{\varphi }(a_{1},\ldots , \widehat{a_{i}},
\ldots , a_{r+1})}
\end{equation*}
where the product is over all $i$ such that $\tilde{\varphi }(a_{1},
\ldots , \widehat{a_{i}},\ldots , a_{r+1})\neq 0$. For each
$e\in E$ define $\lambda _{e}$ by
\begin{equation*}
\lambda _{e}=
\begin{cases}
\frac{ \varphi _{0}(a_{1},\ldots , \widehat{a_{i}},\ldots , a_{r+1})}{
\tilde{\varphi }(a_{1},\ldots , \widehat{a_{i}},\ldots , a_{r+1})}
&
\text{ if $e=a_{i}$ and $\tilde{\varphi }(a_{1},\ldots , \widehat{a_{i}},\ldots , a_{r+1})\neq 0$}
\\
1
&\mbox{ otherwise}
\end{cases}
\end{equation*}
Define $\varphi : E^{r}\to F$ by
\begin{equation*}
\varphi (e_{1}, \ldots , e_{r})=D^{-1}(\bigodot _{i=1}^{r} \lambda _{e
_{i}})\varphi _{0}(e_{1}, \ldots , e_{r})
\end{equation*}
One easily checks that $\varphi $ and $\tilde{\varphi }$ coincide on
$A^{r}$. The Grassmann--Pl\"{u}cker relations for $\varphi _{0}$ imply the
Grassmann--Pl\"{u}cker relations for $\varphi $; thus $\varphi $ is an
$F$-matroid. Also, the functions $\varphi $ and $\varphi _{0}$ have the
same zeroes, so $[\varphi ]\in \operatorname{Real}_{F}^{*}(M)$.
\end{proof}

\begin{lemma}
\label{lem:open}
Let $F$ be a topological hyperfield, $E$ a finite set, and $B$ an
ordered $r$-tuple from $E$.
\begin{enumerate}[\textit{2.}]%
\item[\textit{1.}]
$\operatorname{Gr}^{*}(r, F^{E})_{B} =\{ [\varphi ] \in
\operatorname{Gr}^{*}(r,F^{E}) : \varphi (B) \neq 0\}$ is an open
subset of $\operatorname{Gr}^{*}(r, F^{E})$.
\item[\textit{2.}]
For any $A\supseteq B$, the restriction map $\operatorname{Gr}^{*}(r,
F^{E})_{B}\to \operatorname{Gr}^{*}(r, F^{A})_{B}$ is an open map.
\end{enumerate}
\end{lemma}

\begin{proof}
1. Since $F \backslash \{0\}$ is open, $\{\varphi \in F^{E^{r}} :
\varphi (B) \neq 0\}$ is open in $F^{E^{r}}$. Proposition~\ref{open_map} says that $F^{E^{r}} \backslash \{0\} \to \mathbf{P}(F
^{E^{r}})$ is an open map; thus $\{[\varphi ] \in \mathbf{P}(F^{E^{r}})
: \varphi (B) \neq 0\}$ is open in $\mathbf{P}(F^{E^{r}})$. The result
follows by intersecting with the Grassmannian.

2. Projection maps are open, and so the map $\{\varphi \in F^{E^{r}}:
\varphi (B) \neq 0\}\to \{\varphi \in F^{A^{r}}: \varphi (B) \neq 0\}$
is open. Thus the map $\{[\varphi ]\in \mathbf{P}(F^{E^{r}}):\varphi
(B)\neq 0\}\to \{[\varphi ]\in \mathbf{P}(F^{A^{r}}): \varphi (B)
\neq 0\}$ is also open. The restriction map $\operatorname{Gr}^{*}(r,
F^{E})_{B}\to \operatorname{Gr}^{*}(r, F^{A})_{B}$ is just a restriction
of this map.
\end{proof}

Now we can prove Proposition~\ref{prop:0open}. Consider a (strong or
weak) $F$-matroid $M$ with Grassmann--Pl\"{u}cker function $\varphi :E
^{r}\to F$. A subset $A$ of $E$ is said to have rank $r$ if the
restriction of $\varphi $ to $A^{r}$ is nonzero. In this case this
restriction is the Grassmann--Pl\"{u}cker function of a (strong
resp. weak) $F$-matroid on $A$, denoted $M(A)$ and called the
restriction of $M$ to $A$. (Aside: restriction can also be defined for
subsets $A$ of $E$ of smaller rank. Such a definition, given in terms
of circuits rather than Grassmann--Pl\"{u}cker functions, appears
in~\cite{Baker-Bowler}, and its equivalence with our definition
can be pieced together from several results
in~\cite{Baker-Bowler}. A more direct explanation of the
equivalence of circuit and Grassmann--Pl\"{u}cker definitions of
restriction is given by Proposition~3.4 in~\cite{EJS}.)

\begin{proof}[Proof of Proposition~\ref{prop:0open}]
If $M$ is a rank $r$ matroid on elements $E$ and $\mathcal B$ is the set of
ordered bases of $M$ then
\begin{equation*}
\operatorname{Real}^{*}_{F}(M)=\operatorname{Gr}^{*}(r, F^{E}) \cap
\bigcap _{B\in {\mathcal B}} \{[\varphi ] : \varphi (B)\neq 0\}
\cap
\bigcap _{B\notin {\mathcal B}} \{[\varphi ]: \varphi (B) = 0\}.
\end{equation*}
Since $F^{\times }$ is open in $F$, the implication (a)$\Rightarrow $(b)
is clear. The equivalence of (b) and (c) is also clear.

To show that (c)$\Rightarrow $(d) it suffices to show that for any
hyperfield $F$, there is a nonuniform matroid $M$ with $
\operatorname{Real}_{F}^{*}(M)$ nonempty. For any hyperfield $F$, an
element of $\operatorname{Gr}^{*}(1,F^{2})$ is given by the
Grassmann--Pl\"{u}cker relation $\varphi $ with $\varphi (1) = 1$ and
$\varphi (2) = 0$. Then $M = \operatorname{Gr}^{*}(\kappa ) [\varphi
]$ is a nonuniform matroid. $\operatorname{Real}_{F}^{*}(M)$ contains
$[\varphi ]$ and hence is nonempty.

To see (d)$\Rightarrow $(a), consider a nonuniform rank $r$ matroid
$M$ such that $\operatorname{Real}^{*}_{F}(M)$ is open in $
\operatorname{Gr}^{*}(r, F^{E})$ and is nonempty. Since $M$ is not
uniform, there is an $(r+1)$-element subset $A = \{a_{1}, \dots , a
_{r+1}\}$ of $M$ such that $A$ contains a basis but not every
$r$-element subset of $A$ is a basis. Thus $M(A)$ is not uniform. Let
$B$ be an ordered basis of $M(A)$.

Consider the commutative diagram
\begin{equation*}
\begin{tikzcd}
\operatorname{Real} ^*_F(M) \arrow[d,"\rho _M"] \arrow[r,hook] & \operatorname{Gr} ^*(r,F^E)_B
\arrow[d,"\rho _B"] \arrow[r,hook] & \operatorname{Gr} ^*(r,F^E) \\
\operatorname{Real} ^*_F(M(A)) \arrow[r,hook] & \operatorname{Gr} ^*(r,F^A)_B \arrow[r,hook] &
\operatorname{Gr} ^*(r,F^A)
\end{tikzcd}
\end{equation*}
where the vertical maps are the restriction maps. Since $
\operatorname{Real}^{*}_{F}(M)$ is open in $\operatorname{Gr}^{*}(r,F
^{E})$, Lemma~\ref{lem:open} shows that $\rho _{B}(\operatorname{Real}
^{*}_{F}(M))$ is open in $\operatorname{Gr}^{*}(r,F^{A})_{B}$, and hence
also in $\operatorname{Gr}^{*}(r,F^{A})$. Lemma~\ref{lemma:_prescribe}
implies that $\rho _{M}$ is surjective. Thus $\operatorname{Real}_{F}
^{*}(M(A))$ is open in $\operatorname{Gr}^{*}(r, F^{A})$.

Now consider the identification of $\operatorname{Gr}(r, F^{A})$ with
$F\mathbf{P}^{r}$ given by Lemma~\ref{lem:rplus1}. Under this
identification $\operatorname{Real}_{F}^{*}(M(A))$ is identified with
\begin{equation*}
\{[\vec{x}] \in F\mathbf{P}^{r} : x_{i} = 0
\text{ if and only if $A-\{a_{i}\}$ is not a basis of $M(A)$}\}.
\end{equation*}
Since $M(A)$ is not uniform, this set is open in $F\mathbf{P}^{r}$ if
and only if $\{0\}$ is open in $F$.
\end{proof}

\begin{remark}
A recent paper of Emanuele Delucchi, Linard Hoessly and Elia Saini
\cite{Delucchi-Saini} examines realization spaces in more depth,
including issues specific to hyperfields with various algebraic
properties.
\end{remark}

\section{Poset hyperfields}%
\label{sec:posets}

A poset $P$ can be given the \emph{upper order ideal topology}, or
\emph{poset topology}, which is the topology generated by sets of the
form $U_{p}=\{x\in P: x\geq p\}$, where $p$ is an element of~$P$. Note
that the partial order can be recovered from this topology:
$U_{p}$ is the intersection of all open sets containing $p$, and for any
$p,q\in P$, we have $p\leq q$ if and only if $U_{q}\subseteq U_{p}$.
Topological spaces given by posets are precisely the topological spaces
which are both $T_{0}$ (given any two points, there is an open set
containing exactly one of the points) and Alexandrov (arbitrary
intersections of open sets are open).

If we give $\mathbb{K}$ the partial order in which $1>0$, then the upper
order ideal topology coincides with the 0-coarse topology discussed
earlier. Likewise, if we give $\mathbb{S}$ the partial order with
$+>0$, $->0$, and $+$ incomparable to $-$, then the upper order ideal
topology coincides with the topology on $\mathbb{S}$ already introduced.
More generally, if $F$ is a hyperfield with no endowed topology, we can
give $F$ a partial order by making $0$ the unique minimum and all other
elements incomparable: the resulting poset topology coincides with
$F_{0}$. This partial order induces a partial order on $
\operatorname{Gr}^{*}(r, F^{n})$ in which $[\varphi ]\leq [\varphi ']$
if and only if $\varphi (i_{1},\ldots ,i_{r})\leq \varphi '(i_{1},
\ldots , i_{r})$ for every $i_{1}, \ldots , i_{r}$. In other words,
$[\varphi ]\leq [\varphi ']$ if and only if the 0 set of $\varphi $
contains the 0 set of $\varphi '$; thus the partial order on
$\operatorname{Gr}^{*}(r, F^{n})$ is independent of the choice of
$\varphi $ and $\varphi '$. The poset topology on $\operatorname{Gr}
^{*}(r,F^{n})$ and the Grassmann hyperfield topology on $
\operatorname{Gr}^{*}(r,F_{0})$ coincide.

For matroids ($F = \mathbb{K}$) and oriented matroids ($F =
\mathbb{S}$), the partial order is the well-known \emph{weak map} partial
order. The poset $\operatorname{Gr}(r, \mathbb{S}^{n})$ is known as the
\emph{MacPhersonian} and is denoted $\mathrm{MacP}(r,n)$. Thus we have
an identification of topological spaces
\begin{equation*}
\mathrm{MacP}(r,n) = \operatorname{Gr}(r, \mathbb{S}^{n}).
\end{equation*}

The poset $\operatorname{Gr}(r, \mathbb{K}^{n})$ is contractible: we can
see this by noting:
\begin{enumerate}%
\item
$\operatorname{Gr}(r, \mathbb{K}^{n})$ has a unique maximal element,
given by the uniform rank $r$ matroid on $[n]$, and
\item
If $P$ is a poset with a unique maximum element $\hat{1}$, then the
function $H:P\times I\to P$ given by
\begin{equation*}
H(x,t)=\left \{
\begin{array}{l@{\quad }l}
x&\mbox{ if $t=0$}
\\
\hat{1}&\mbox{ if $t>0$}
\end{array}
\right .
\end{equation*}
is a deformation retract of $P$ to $\{\hat{1}\}$.
\end{enumerate}

There is a second topological space we can associate to a poset
$(P,\leq )$: the order complex $\|P\|$. This is the geometric
realization of the simplicial complex whose vertices are the elements
of $P$ and whose $k$-simplices are the chains $p_{0} < p_{1} < \dots
< p_{k}$ of elements of $P$. The function $\varphi _{P}: \|P\| \to P$
given by sending points in the interior of the simplex spanned by
$p_{0} < p_{1} < \dots < p_{k}$ to $p_{k}$ is a continuous function when
$P$ is given the poset topology.

McCord \cite{mccord} proved the following amazing theorem.

\begin{theorem}
For any poset $P$, the map $\varphi _{P} : \|P\| \to P$ is a weak
homotopy equivalence.
\end{theorem}

In particular, McCord's theorem implies that every finite simplicial
complex has the weak homotopy type of its poset of simplices, giving
a close connection between the topology of finite complexes and finite
topological spaces. Recall that a \emph{weak homotopy equivalence} is a
continuous function $f : X \to Y$ such that
\begin{equation*}
f_{*} : \pi _{n}(X,x_{0}) \to \pi _{n}(Y,f(x_{0}))
\end{equation*}
is a bijection for all $n \geq 0$ and all $x_{0} \in X$. By the Hurewicz
Theorem, a weak homotopy equivalence induces isomorphisms on homology
groups and thus on the cohomology ring.

\begin{remark}
\label{mccordremark}
If $Z$ is a CW-complex and $f : X \to Y$ is a weak homotopy equivalence,
then $f_{*} : [Z,X] \to [Z,Y]$ is a bijection, where $f_{*}(g) = f
\circ g$ and $[-,-]$ denotes homotopy classes of maps.
\end{remark}

McCord's theorem and the above identification of $\mathrm{MacP}(r,n)$
with $\operatorname{Gr}(r,\mathbb{S}^{n})$ give the first part of
Theorem~\ref{cohomology}. In Section~\ref{sec:macphersonian}, we will
review what is known about the homotopy type of $\|\mathrm{MacP}(r,n)
\|$.

\section{The MacPhersonian and hyperfields}%
\label{sec:macphersonian}

In this section we interpret our previous work
\cite{Anderson-Davis} on the relationship between the Grassmannian and
the MacPhersonian in light of the continuous hyperfield homomorphism
$\mathbb{R}\to \mathbb{S}$. As noted before, $\operatorname{MacP}(r,n)$
with the poset topology is just $\operatorname{Gr}(r,\mathbb{S}^{n})$
with the topology induced by the topological hyperfield $\mathbb{S}$.
The homomorphisms of topological hyperfields
\begin{equation*}
\mathbb{R}\to \mathbb{S}\to \mathcal{T }\mathbb{R}_{0}
\end{equation*}
induce continuous maps
\begin{equation*}
\operatorname{Gr}(r,\mathbb{R}^{n}) \xrightarrow{\mu }
\operatorname{Gr}(r,\mathbb{S}^{n}) \to \operatorname{Gr}(r,\mathcal{T
}\mathbb{R}_{0}^{n}).
\end{equation*}
By Corollary~\ref{cor:grassmannhomotopy} we know that the second map is
a homotopy equivalence.

Here is the main theorem of \cite{Anderson-Davis}, whose proof
involves construction of a universal matroid spherical quasifibration
over $\|\mathrm{MacP}(r,n)\|$ and its Stiefel--Whitney classes.

\begin{theorem}[\cite{Anderson-Davis}]
\label{AD}
There is a continuous map ${\widetilde{\mu }}: \operatorname{Gr}(r,
\mathbb{R}^{\infty }) \to \|\operatorname{MacP}(r,\infty )\|$ such that
\begin{enumerate}[\textit{2.}]%
\item[\textit{1.}]
${\widetilde{\mu }}^{*} : H^{*}(\|\mathrm{MacP}(r,\infty )\|;
\mathbb{F}_{2}) \to H^{*}(\operatorname{Gr}(r,\mathbb{R}^{\infty });
\mathbb{F}_{2})$ is a split epimorphism of graded rings.
\item[\textit{2.}]
The diagram below commutes
\begin{equation*}
\begin{tikzcd}
&\|\operatorname{MacP} (r,\infty )\| \arrow[d, "\varphi _{\operatorname{MacP} (r,\infty )}"]\\
\operatorname{Gr} (r,\mathbb R^\infty ) \arrow[r, "\mu "] \arrow[ru,
"{\widetilde{\mu }}"] & \operatorname{MacP} (r, \infty ).
\end{tikzcd}
\end{equation*}
\end{enumerate}
\end{theorem}

The existence of a map ${\widetilde{\mu }}$ making the above triangle
commute up to homotopy follows from Remark~\ref{mccordremark}.

The following corollary finishes the proof of Theorem~\ref{cohomology}.

\begin{corollary}
Let $\operatorname{ph}: \mathbb{R}\to \mathbb{S}$ be the morphism of
topological hyperfields given by the phase map $x \mapsto x/|x|$ for
$x \neq 0$.
\begin{enumerate}[\textit{2.}]%
\item[\textit{1.}]
$\operatorname{Gr}(\operatorname{ph})^{*} : H^{*}(\operatorname{Gr}(r,
\mathbb{S}^{\infty });\mathbb{F}_{2}) \to H^{*}(\operatorname{Gr}(r,
\mathbb{R}^{\infty });\mathbb{F}_{2}) $ is a split epimorphism of graded
rings.
\item[\textit{2.}]
$\operatorname{Gr}(\operatorname{ph})^{*} : H^{*}(\operatorname{Gr}(r,
\mathbb{S}^{n});\mathbb{F}_{2}) \to H^{*}(\operatorname{Gr}(r,
\mathbb{R}^{n});\mathbb{F}_{2}) $ is a epimorphism of graded rings.
\end{enumerate}
\end{corollary}

\begin{proof}
Part 1 follows from the identification $\mathrm{MacP}(r,\infty ) =
\operatorname{Gr}(r,\mathbb{S}^{\infty })$, Theorem~\ref{AD}, and
McCord's theorem. Part 2 follows from considering the induced maps in
mod 2 cohomology from the following commutative diagram
%
\begin{equation}
\begin{CD}
\label{square}
\operatorname{Gr}(r,\mathbb{R}^{n}) @>\mu _{n}>> \operatorname{Gr}(r,
\mathbb{S}^{n})
\\
@VVi^{n}_{\mathbb{R}}V @VVi^{n}_{\mathbb{S}}V
\\
\operatorname{Gr}(r,\mathbb{R}^{\infty }) @>\mu _{\infty } >>
\operatorname{Gr}(r,\mathbb{S}^{\infty })
\end{CD}
\end{equation}
In particular the map $\mu _{\infty }$ induces a surjection on mod 2
cohomology by Theorem~\ref{AD} above. The map $i^{n}_{\mathbb{R}}$
induces a surjection on mod 2 cohomology since the mod 2 cellular chain
complex of $\operatorname{Gr}(r,\mathbb{R}^{\infty })$ has zero
differentials when the cell structure is given by Schubert cells (see
\cite{Milnor-Stasheff}), and $ \operatorname{Gr}(r,\mathbb{R}
^{n})$ is a subcomplex.
\end{proof}

Finally, we discuss the maps induced on mod 2 cohomology by square \eqref{square} because we feel that there are issues of combinatorial
interest. First, we don't know if the map $\mu _{\infty }$ induces an
injection on mod 2 cohomology; elements in the kernel would be exotic
characteristic classes for matroid bundles
\cite{Anderson-Davis}. Next we discuss the vertical maps. We will use
a result of \cite{Anderson:homotopy} which states that the
homotopy groups of $\|\operatorname{MacP}(r,n)\|$ are stable in~$n$;
more precisely that $\pi _{j}\| \operatorname{MacP}(r, n) \| \to \pi
_{j+1}\| \operatorname{MacP}(r, n+1)\|$ is an isomorphism for
$n \geq r(j+2)$.

\begin{proposition}
For any $j$ and $r$, there exists $n(j,r)$ such that for $n > n(j,r)$, the maps
\begin{align*}
H^{j}(\operatorname{Gr}(r,\mathbb{R}^{\infty });\mathbb{F}_{2})
&
\to H^{j}(\operatorname{Gr}(r,\mathbb{R}^{n});\mathbb{F}_{2})
\\
H^{j}(\operatorname{Gr}(r,\mathbb{S}^{\infty });\mathbb{F}_{2})
&
\to H^{j}(\operatorname{Gr}(r,\mathbb{S}^{n});\mathbb{F}_{2})
\end{align*}
are isomorphisms.
\end{proposition}

\begin{proof}
For the real Grassmannian, this is classical
\cite{Milnor-Stasheff}. Indeed, the CW-complex $\operatorname{Gr}(r,
\mathbb{R}^{\infty })$ is the union of the subcomplexes $
\operatorname{Gr}(r,\mathbb{R}^{n})$ and the dimension of every cell of
$\operatorname{Gr}(r,\mathbb{R}^{n+1}) - \operatorname{Gr}(r,
\mathbb{R}^{n})$ is greater than $n-r$. It thus follows from using
cellular cohomology.

For the sign hyperfield, the reasoning is more subtle. First, using
McCord's Theorem, it suffices to prove the result with $
\operatorname{Gr}(r,\mathbb{S}^{-})$ replaced by $\|
\operatorname{MacP}(r, -) \|$. Any map from a compact set to
$\| \operatorname{MacP}(r, \infty ) \|$ must land in $\|
\operatorname{MacP}(r, N) \|$ for some $N$. Thus the result of
\cite{Anderson:homotopy} shows that there is an $n(j,r)$ such that for
$n > n(j,r)$, the maps $\pi _{j} \| \operatorname{MacP}(r, n) \|
\to \pi _{j} \| \operatorname{MacP}(r, \infty ) \|$ is an isomorphism.
The Relative Hurewicz Theorem and Universal Coefficient Theorem imply
the same is true on mod 2 cohomology.
\end{proof}

For an arbitrary hyperfield $F$, we do not know, for example, whether
the map $H^{*}(\operatorname{Gr}(r,F^{\infty })) \to \text{lim}_{n
\to \infty } H^{*}(\operatorname{Gr}(r,F^{n}))$ is an isomorphism.

Similar considerations allow us to deduce the following from the main
results of~\cite{Anderson:homotopy}.
%
\begin{corollary}
\begin{enumerate}[\textit{2.}]%
\item[\textit{1.}]
$\operatorname{Gr}(\operatorname{ph})_{*}:\pi _{i}(\operatorname{Gr}(r,
\mathbb{R}^{n}))\to \pi _{i}(\operatorname{Gr}(r, \mathbb{S}^{n}))$ is
an isomorphism for $i\in \{0,1\}$ and a surjection for $i=2$.
\item[\textit{2.}]
The map $\pi _{i}(\operatorname{Gr}(r,\mathbb{S}^{n-1}))\to \pi _{i}(
\operatorname{Gr}(r,\mathbb{S}^{n}))$ induced by the stabilization
embedding is an isomorphism if $n>r(i+2)$ and a surjection if
$n>r(i+1)$.
\end{enumerate}
\end{corollary}

We now wish to discuss the map induced by $ \operatorname{Gr}(r,
\mathbb{R}^{n}) \to \operatorname{Gr}(r,\mathbb{S}^{n})$ on mod 2
cohomology. There is a factorization
\begin{equation*}
\mathbb{F}_{2}[w_{1}, \dots , w_{r}] \xrightarrow{\alpha } H^{*}(
\operatorname{Gr}(r,\mathbb{S}^{n});\mathbb{F}_{2})
\xrightarrow{\beta } H^{*}(\operatorname{Gr}(r,\mathbb{R}^{n});
\mathbb{F}_{2})
\end{equation*}
where the first map is given by the Stiefel--Whitney classes of the
universal matroid spherical quasifibration of
\cite{Anderson-Davis}. Note that $\beta \circ \alpha $ is onto by
problem 7-B of \cite{Milnor-Stasheff}. We then can break the
question of whether $\beta $ is injective into two pieces. First, is
$\alpha $ onto, i.e.~are there exotic characteristic classes? Second,
is $\beta $ restricted to the image of $\alpha $ injective, i.e.~do the
relations in the mod 2 cohomology of the real Grassmannian hold over
$\mathbb{S}$?

\bibliographystyle{alpha}
\bibliography{ad}

\Addresses

\end{document}